\documentclass{article}%
\usepackage{amsmath}
\usepackage{graphicx}
\usepackage{amsfonts}
\usepackage{amssymb}%
\providecommand{\U}[1]{\protect\rule{.1in}{.1in}}
\providecommand{\U}[1]{\protect \rule{.1in}{.1in}}
\providecommand{\U}[1]{\protect \rule{.1in}{.1in}}

\newtheorem{theorem}{Theorem}

\newtheorem{definition}[theorem]{Definition}
\newtheorem{example}[theorem]{Example}

\newtheorem{lemma}[theorem]{Lemma}

\newtheorem{proposition}[theorem]{Proposition}

\newenvironment{proof}[1][Proof]{\noindent \textbf{#1.} }{\hfill  \rule{0.5em}{0.5em}}
\begin{document}

\title{Affine Gateaux Differentials and \\the von Mises Statistical Calculus}
\author{Simone Cerreia-Vioglio\thanks{Department of Decision Sciences and IGIER,
Università Bocconi, Via Roentgen, 1, Milan, Italy. E-mail:
simone.cerreia@unibocconi.it, fabio.maccheroni@unibocconi.it,
massimo.marinacci@unibocconi.it}
\and Fabio Maccheroni\footnotemark[1]
\and Massimo Marinacci\footnotemark[1]
\and Luigi Montrucchio\thanks{ESOMAS Department and Collegio Carlo Alberto,
Università di Torino, Corso Unione Sovietica, 218 Bis 10134, Turin,
Italy. Emails: luigi.montrucchio@unito.it, lorenzomaria.stanca@unito.it}
\and Lorenzo Stanca\footnotemark[2] }
\maketitle

\begin{abstract}
This paper presents a general study of one-dimensional differentiability for
functionals defined on convex domains that are not necessarily open. The local
approximation is carried out using affine functionals, as opposed to linear
functionals typically employed in standard Gateaux differentiability. This
affine notion of differentiability naturally arises in certain applications
and has been utilized by some authors in the statistics literature. We aim to
offer a unified and comprehensive perspective on this concept.

\end{abstract}



\section{Introduction}

To study the asymptotic behavior of statistical functionals, Richard von Mises
developed a notion of directional derivability for functionals defined over
spaces of probability distributions.\footnote{See \cite[Part II]{mises} as
well as Reeds \cite{reed} and Fernholz \cite{fern}.} Since these domains lack
interior points, standard Gateaux differentiability is no longer adequate.

In this paper, we build upon von Mises's idea by studying functionals defined
over abstract convex sets. For a function $f:C\rightarrow\mathbb{R}$ defined
on a convex set $C$, the starting point is the directional derivative
\[
Df\left(  x;y\right)  =\lim_{t\downarrow0}\frac{f\left(  \left(  1-t\right)
x+ty\right)  -f\left(  x\right)  }{t}%
\]
at a point $x$ of $C$ along the direction $y$, where $y$ is another point of
$C$. The increment $f\left(  \left(  1-t\right)  x+ty\right)  -f\left(
x\right)  $ equals $f\left(  x+t\left(  y-x\right)  \right)  -f\left(
x\right)  $. Consequently, by the convexity of the domain $C$, the directional
derivative $Df\left(  x;y\right)  $ is essentially the classical directional
derivative restricted to the cone of feasible directions at the point $x$. We
use the term weak affine differential when the function $Df(x;\cdot)$ is
affine on $C$ (Definition \ref{def:aff}). For this basic notion, it is already
possible to prove a number of results. Yet, we reserve the term affine
differential for the important special case when the affine functional
$Df\left(  x;\cdot\right)  $ can be extended to the whole space, yielding a
notion of affine gradient. For example, the von Mises differential is simply
our affine differential for functionals whose domain is a space of probability distributions.

The early notion of differentiability introduced by von Mises has been used in
statistics (see Hampel \cite{hamp}, Huber \cite[pp.~34-40]{huber}, and
Fernholz \cite{fern}). A similar approach, again on spaces of distributions,
has also been used in risk theory (see Chew et al. \cite{chew1}). More
recently, a general formulation of Roy's identity in consumer theory has been
established through affine derivatives by \cite{green}. In summary, affine
differentiability naturally appears in several applications. Our purpose is to
provide a systematic analysis of this notion.

\paragraph{Outline}

In Section \ref{ch1}, we introduce affine differentiability. Notably, we
demonstrate that the Gateaux gradient and the affine gradient coincide only
over the algebraic interior of the domain, a set that is empty in relevant
applications. The main result, Lemma \ref{lmm:ccc}, shows that weakly affine
differentiable functions are hemidifferentiable, that is, differentiable over
line segments. A mean value theorem, Theorem \ref{cor:mean}, follows as a
consequence of this lemma. In Section \ref{ch2}, we discuss applications to
optimization. In particular, we establish a Danskin-Demyanov type envelope theorem.

Applications to economics and statistics are considered in Section \ref{ch3}.
In risk theory, Theorem \ref{th:jop} provides a global perspective to the
local expected utility analysis of Machina \cite{mach} (see also \cite{cmm}).
As an illustration, we compute the local utilities for the quadratic model of
Chew et al. \cite{chew} and for the prospect theory model of Tversky and
Kahneman \cite{kantv}. In a final example, we apply the envelope theorem of
Section \ref{ch2} to a Bayesian statistical problem.

Basic affine differentiation often proves too weak in applications. For this
reason, in Section \ref{ch5}, we present some stronger variants, non-trivial
versions of the classical Hadamard and Frechet differentiation in normed
vector spaces. We illustrate these notions with applications in risk theory
and statistics.

\section{Preliminaries}

Throughout $\left\langle X,X^{\ast}\right\rangle $ is a dual pair between two
vector spaces $X$ and $X^{\ast}$, with generic elements $x$ and $x^{\ast}$.
When $X$ is normed, $X^{\ast}$ is its topological dual, unless otherwise
stated. The pairing map is denoted by $\left\langle x,x^{\ast}\right\rangle $,
with $x\in X$ and $x^{\ast}\in X^{\ast}$. We refer to the $\sigma\left(
X,X^{\ast}\right)  $-topology as the weak topology of $X$, while we refer to
the $\sigma\left(  X^{\ast},X\right)  $-topology as the weak* topology of
$X^{\ast}$. Let $A$ be a subset of $X$. Its \emph{annihilator}, denoted by
$A^{\bot}$, is the set of all $x^{\ast}\in X^{\ast}$ that vanish on $A$, i.e.,
$\left\langle x,x^{\ast}\right\rangle =0$ for all $x\in A$. Clearly, $A^{\bot
}$ is a weak*-closed vector subspace of $X^{\ast}$.

Throughout $C$ denotes a convex subset of $X$. A point $x\in C$ is an
\emph{algebraic interior point} of $C$ if for every $y\in X$ there is
$\varepsilon>0$ such that $x+\varepsilon y\in C$. The \emph{algebraic
interior} of $C$ is denoted by $\operatorname*{cor}C$. A point $x\in C$ is an
(\emph{algebraic}) \emph{relative} \emph{interior point} of $C$ if, for every
$y\in\operatorname*{aff}C$ there is $\varepsilon>0$ such that $x+\varepsilon
y\in C$. The \emph{relative interior} of $C$ is denoted by $\operatorname*{ri}%
C$. For short, we call \emph{internal} the elements of $\operatorname*{ri}C$.
Clearly, $\operatorname*{cor}C\subseteq\operatorname*{ri}C$.

Let $Y$ be a topological space, typically assumed to be metrizable. The set of
all finite signed Borel measures is denoted by $ca\left(  Y\right)  $. The
subset $\Delta(Y)$ of $ca\left(  Y\right)  $ consists of all Borel probability
measures. We denote by $C_{b}\left(  Y\right)  $ the space of all bounded and
continuous functions on $Y$.

A function $f:C\rightarrow\mathbb{R}$ is \emph{affine} if $f\left(  tx+\left(
1-t\right)  y\right)  =tf\left(  x\right)  +\left(  1-t\right)  f\left(
y\right)  $ for all $x,y\in C$ and all $t\in\left[  0,1\right]  $.

\begin{definition}
An affine function $f:C\rightarrow\mathbb{R}$ is \emph{extendable} when it
admits a weakly continuous linear extension to the whole space $X$, that is,
when there exist $x^{\ast}\in X^{\ast}$ and $\gamma\in\mathbb{R}$ such that
$f\left(  \cdot\right)  =\left\langle \cdot,x^{\ast}\right\rangle +\gamma$.
\end{definition}

As is well-known, affine functionals on $C$ are not always extendable, even
when they are weakly continuous (cf. Example \ref{ex:hilb}). Furthermore, when
an extension of an affine functional does exist, it may not be unique.

\begin{proposition}
\label{prop:aff}Let $f\left(  \cdot\right)  =\left\langle \cdot,x^{\ast
}\right\rangle +\gamma$ on $X$, with $\gamma\in\mathbb{R}$. An affine
functional $g:X\rightarrow\mathbb{R}$ agrees with $f$ on $C$ if and only if
$g\left(  \cdot\right)  =\left\langle \cdot,y^{\ast}\right\rangle
+\gamma_{y^{\ast}}$, where $y^{\ast}\in x^{\ast}+\left(  C-C\right)  ^{\bot}$
and $\gamma_{y^{\ast}}=f\left(  \bar{x}\right)  -\left\langle \bar{x},y^{\ast
}\right\rangle $, with $\bar{x}\in C$ arbitrarily fixed.
\end{proposition}

\noindent\textbf{Proof} Let $g\left(  \cdot\right)  =\left\langle
\cdot,y^{\ast}\right\rangle +\gamma_{1}$ be an affine functional on $X$ such
that, for each $x\in C$, $f\left(  x\right)  =g\left(  x\right)  $. This
implies that $\left\langle x,x^{\ast}-y^{\ast}\right\rangle $ is constant when
$x$ runs over $C$. That is, $\left\langle x-y,y^{\ast}-x^{\ast}\right\rangle
=0$ when $x$ and $y$ are distinct points of $C$. Hence, $y^{\ast}-x^{\ast}%
\in\left(  C-C\right)  ^{\bot}$. Equivalently, $y^{\ast}\in x^{\ast}+\left(
C-C\right)  ^{\bot}$. From $f\left(  \bar{x}\right)  =\left\langle \bar
{x},y^{\ast}\right\rangle +\gamma_{1}$ it follows that $\gamma_{1}=f\left(
\bar{x}\right)  -\left\langle \bar{x},y^{\ast}\right\rangle $.

To prove the converse, let $g\left(  \cdot\right)  =\left\langle \cdot
,y^{\ast}\right\rangle +\gamma_{y^{\ast}}$ be with $y^{\ast}$ and
$\gamma_{y^{\ast}}$ given above. Tedious algebra shows that%
\[
g\left(  x\right)  =\left\langle x-\bar{x}+\bar{x},y^{\ast}-x^{\ast}+x^{\ast
}\right\rangle +f\left(  \bar{x}\right)  -\left\langle \bar{x},y^{\ast
}\right\rangle =\left\langle x,x^{\ast}\right\rangle +\gamma=f\left(
x\right)
\]
as desired.\hfill$\blacksquare$

\bigskip

Thus, the representation of an affine functional on $C$ is unique if and only
if $\left(  C-C\right)  ^{\bot}=\left\{  0\right\}  $. Since%
\begin{equation}
\left(  C-C\right)  ^{\bot}=\left(  C-x\right)  ^{\bot} \label{eq:kkk}%
\end{equation}
holds for any arbitrary point $x\in C$, we have $\left(  C-C\right)  ^{\bot
}=\left\{  0\right\}  $ when $\operatorname*{cor}C\neq\emptyset$. Another
example, which gives a unique extension (if any), is the case of an affine
functional defined on a reproducing convex cone $K$. Clearly $\left(
K-K\right)  ^{\bot}=X^{\bot}=\left\{  0\right\}  $.

Uniquely extendable affine functionals can be characterized in
finite-dimensional spaces, as the next more or less known result shows.

\begin{proposition}
\label{prop:zac}An affine functional $f:C\subseteq\mathbb{R}^{n}%
\rightarrow\mathbb{R}$ is extendable. Its extension is unique if and only if
$\operatorname*{int}C\neq\emptyset$.
\end{proposition}

In normed vector spaces a much less general result holds.

\begin{proposition}
\label{prop:xet}An affine functional $f:C\rightarrow\mathbb{R}$ defined on a
convex set $C$ of a normed vector space $X$ is uniquely extendable if
$\operatorname*{int}C\neq\emptyset$. The extension is weakly continuous when
$f$ is locally bounded around to some point of $\operatorname*{int}C$.
\end{proposition}

\section{Affine differentiability\label{ch1}}

\subsection{Differential}

We begin with the protagonist of our analysis.

\begin{definition}
\label{def:dir}The \emph{affine directional} \emph{derivative} of a functional
$f:C\rightarrow\mathbb{R}$ at a point $x\in C$ along the point $y\in C$ is
given by
\[
Df\left(  x;y\right)  =\lim_{t\downarrow0}\frac{f\left(  \left(  1-t\right)
x+ty\right)  -f\left(  x\right)  }{t}%
\]
when this limit exists and is finite.
\end{definition}

When this limit exists finite for all $y\in C$, the map $Df\left(
x;\cdot\right)  :C\rightarrow\mathbb{R}$ is well defined and satisfies the
following properties:

\begin{enumerate}
\item[(i)] $Df\left(  x;x\right)  =0$;

\item[(ii)] $Df\left(  x;\cdot\right)  $ is homogeneous, i.e.,
\begin{equation}
Df\left(  x;\left(  1-\alpha\right)  x+\alpha y\right)  =\alpha Df\left(
x;y\right)  \label{eq:sat}%
\end{equation}
for all $y\in C$ and all $\alpha\in\left[  0,1\right]  $.
\end{enumerate}

The map $Df\left(  x;\cdot\right)  $ is, in general, not affine (see Example
\ref{ex:nonaff} below), a failure that motivates the following taxonomy.

\begin{definition}
\label{def:aff}A functional $f:C\rightarrow\mathbb{\mathbb{R}}$ is:

\begin{enumerate}
\item[(i)] \emph{weakly} \emph{affinely differentiable }at\emph{ }$x\in C$, if
$Df\left(  x;\cdot\right)  :C\rightarrow\mathbb{R}$ is affine;

\item[(ii)] \emph{affinely differentiable }if $Df\left(  x;\cdot\right)
:C\rightarrow\mathbb{R}$ is an extendable affine functional.
\end{enumerate}
\end{definition}

Weak affine differentiability (for short, wa-differentiability) does not
require any vector topology, which is instead needed for affine
differentiability (for short, a-differentiability). We call the map $Df\left(
x;\cdot\right)  :C\rightarrow\mathbb{R}$ the \emph{wa-differential} of $f$ at
$x$ when it is affine.

When $Df\left(  x;\cdot\right)  $ is extendable, we call it the
\emph{a-differential} of $f$ at $x$; in this case, there is a pair $\left(
x^{\ast},\gamma\right)  \in X^{\ast}\times\mathbb{R}$ such that%
\begin{equation}
Df\left(  x;\cdot\right)  =\left\langle \cdot,x^{\ast}\right\rangle
+\gamma\label{eq:hup}%
\end{equation}
This pair is not unique unless $\left(  C-C\right)  ^{\bot}=\left\{
0\right\}  $. Since $Df\left(  x;x\right)  =0$, it follows that $\gamma
=-\left\langle x,x^{\ast}\right\rangle $ and so the affine differential admits
the intrinsic representation
\begin{equation}
Df\left(  x;y\right)  =\left\langle y-x,x^{\ast}\right\rangle \label{eq:hju}%
\end{equation}
where the inessential scalar $\gamma$ has been dropped. In light of
(\ref{eq:kkk}), equation (\ref{eq:hju}) is independent of the element
$x^{\ast}$ chosen. With this, we call \emph{affine gradient }of\emph{ }$f$ at
$x$ any such equivalent $x^{\ast}$. As it is unique up to elements in $\left(
C-C\right)  ^{\bot}$, we have an equivalence class%
\begin{equation}
\left[  \nabla_{a}f\left(  x\right)  \right]  =x^{\ast}+\left(  C-C\right)
^{\perp}=x^{\ast}+\left(  C-x\right)  ^{\perp}\in X^{\ast}/\left(  C-C\right)
^{\perp} \label{eq:kuku}%
\end{equation}
which is a weak*-closed affine subspace of $X^{\ast}$. Here $\nabla
_{a}f\left(  x\right)  $ is a representative element of this equivalent class,
so it stands for any element $x^{\ast}$ of this equivalence class. With this,
we can write (\ref{eq:hju}) as%
\begin{equation}
Df\left(  x;y\right)  =\left\langle y-x,\nabla_{a}f\left(  x\right)
\right\rangle \label{eq:FFT}%
\end{equation}

\begin{example}
\label{ex:nonaff}\emph{Define} $f:\mathbb{R}_{+}^{2}\rightarrow\mathbb{R}$
\emph{by} $f\left(  x_{1},x_{2}\right)  =x_{1}^{\alpha}x_{2}^{\beta}$\emph{,
with }$\alpha,\beta>0$. \emph{When }$\alpha+\beta<1$\emph{,} \emph{the affine
directional derivative at the origin }$\mathbf{0}=\left(  0,0\right)  $\emph{
does not exist. Instead, it exists when }$\alpha+\beta=1$\emph{, with
}$Df\left(  \mathbf{0};y_{1},y_{2}\right)  =f\left(  y_{1},y_{2}\right)  $.
\emph{But, }$f$\emph{ is not wa-differentiable at} $\mathbf{0}$ \emph{since
}$Df\left(  \mathbf{0};\cdot\right)  $\emph{ is not affine. It is easy to see
that }$f$\emph{ is a-differentiable at }$\mathbf{0}$\emph{ when }$\alpha
+\beta>1$\emph{ and the zero functional is} \emph{its unique gradient}%
.\hfill$\blacktriangle$
\end{example}

It is desirable to have criteria ensuring a-differentiability. The following
conditions are a direct consequence of Propositions \ref{prop:zac} and
\ref{prop:xet}.

\begin{proposition}
\label{prop:has}Let $f:C\rightarrow\mathbb{R}$ be wa-differentiable at a point
$x\in C$. Then, $f$ is a-differentiable at $x$ if either (i) $X=\mathbb{R}%
^{n}$ or (ii) $X$ is normed, $x\in\operatorname*{int}C$, and $f$ is locally
Lipschitz at $x$.
\end{proposition}

\noindent\textbf{Proof} The case $X=\mathbb{R}^{n}$ follows from Proposition
\ref{prop:zac}. When $X$ is normed, the local Lipschitz condition at
$x\in\operatorname*{int}C$ implies that $\left\vert Df\left(  x;y\right)
\right\vert \leq L\left\Vert y-x\right\Vert $ and therefore $Df\left(
x;\cdot\right)  $ is locally bounded at $x$. Therefore, Proposition
\ref{prop:xet} yields the desired result.\hfill$\blacksquare$

\bigskip

When $f:C\rightarrow\mathbb{R}$ is Gateaux differentiable at $x\in C$, the
\emph{Gateaux gradient }$\nabla_{G}f\left(  x\right)  \in X^{\ast}$ is given
by%
\[
\lim_{t\rightarrow0^{+}}\frac{f\left(  x+ty\right)  -f\left(  x\right)  }%
{t}=\left\langle y,\nabla_{G}f\left(  x\right)  \right\rangle \qquad\forall
y\in X
\]
The next result shows that affine and Gateaux differentiability are equivalent
only on $\operatorname*{cor}C$, a set that in most relevant cases is empty.

\begin{proposition}
\label{equivalence} The function $f:C\rightarrow\mathbb{R}$ is
a-differentiable at $x\in\operatorname*{cor}C$ if and only if it is Gateaux
differentiable at $x$. In this case, $\left[  \nabla_{a}f\left(  x\right)
\right]  \equiv\left\{  \nabla_{G}f\left(  x\right)  \right\}  $.
\end{proposition}

Here the affine gradient inherits the uniqueness of the Gateaux one.

\bigskip

\noindent\textbf{Proof} Let $f$ be Gateaux differentiable at $x\in
\operatorname*{cor}C$ and let $x^{\ast}=\nabla_{G}f\left(  x\right)  $. For
each $y\in C$,%
\[
\left\langle y-x,x^{\ast}\right\rangle =\lim_{t\rightarrow0^{+}}\frac{f\left(
x+t\left(  y-x\right)  \right)  -f\left(  x\right)  }{t}=Df\left(  x;y\right)
\]
Hence, $x^{\ast}\in\left[  \nabla_{a}f\left(  x\right)  \right]  $. As
$x\in\operatorname*{cor}C$, the set $\left[  \nabla_{a}f\left(  x\right)
\right]  $ is a singleton and so we can write $x^{\ast}=\nabla_{a}f\left(
x\right)  $. Conversely, let $f$ be a-differentiable at $x$, with $x^{\ast}%
\in\left[  \nabla_{a}f\left(  x\right)  \right]  $. Take $y\in X$. It follows
that $x+t_{0}y\in C$ for some $t_{0}>0$ small enough. It holds
\[
t_{0}\left\langle v,x^{\ast}\right\rangle =\left\langle t_{0}v,x^{\ast
}\right\rangle =Df\left(  x;x+t_{0}v\right)  =\lim_{t\rightarrow0^{+}}%
\frac{f\left(  \left(  1-t\right)  x+t\left(  x+t_{0}v\right)  \right)  }%
{t}=t_{0}\left\langle v,\nabla_{G}f\left(  x\right)  \right\rangle
\]
Hence, $x^{\ast}=\nabla_{G}f\left(  x\right)  $. This proves that $\left[
\nabla_{a}f\left(  x\right)  \right]  =\left\{  \nabla_{G}f\left(  x\right)
\right\}  $.\hfill$\blacksquare$

\bigskip

Observe that when $f$ has an extension $\tilde{f}:X\rightarrow\mathbb{R}$
which is Gateaux differentiable on $C$, then $f$ is a-differentiable on $C$
and $\nabla_{a}f\left(  x\right)  =\nabla_{G}\tilde{f}\left(  x\right)
+\left(  C-C\right)  ^{\perp}$ for all $x\in C$.

\subsection{Examples \label{sect:examples}}

We present a few examples to illustrate the concepts introduced so far.

\begin{example}
\label{ex:lop}\emph{Consider the dual pair} $\left\langle ca\left(  Y\right)
,C_{b}\left(  Y\right)  \right\rangle $\emph{, where} $Y$\emph{ be a metric
space. In view of (\ref{eq:FFT}), a functional} $T:\Delta\left(  Y\right)
\rightarrow\mathbb{R}$ \emph{is affinely differentiable at }$\mu\in
\Delta\left(  Y\right)  $\emph{ if there is an element }$u_{\mu}\in
C_{b}\left(  Y\right)  $ \emph{such that}%
\[
DT\left(  \mu;\lambda\right)  =\int_{X}u_{\mu}\mathrm{d}\left(  \lambda
-\mu\right)  \qquad\forall\mu\in\Delta\left(  Y\right)
\]
\emph{The gradient }$u_{\mu}\in C_{b}\left(  Y\right)  $\emph{ is unique up to
an additive constant. Indeed, if }$u\in\left(  \Delta\left(  Y\right)
-\Delta\left(  Y\right)  \right)  ^{\perp}$, \emph{then }$\left\langle
u,\lambda-\mu\right\rangle =0$\emph{ for all }$\lambda,\mu\in\Delta\left(
Y\right)  $\emph{. By setting }$\lambda=\delta_{x}$\emph{ and }$\mu=\delta
_{y}$\emph{, we then get }$u\left(  x\right)  =u\left(  y\right)  $ \emph{for
all }$x,y\in Y$\emph{. Hence, }$u$\emph{ is a constant function.}%
\newline\emph{A similar argument applies for the functionals }$T:\mathcal{D}%
\left[  a,b\right]  \rightarrow\mathbb{R}$ \emph{where} $\mathcal{D}\left[
a,b\right]  $ \emph{is the space of the probability distributions on the
interval} $\left[  a,b\right]  $. \emph{The space} $\mathcal{D}\left[
a,b\right]  $ \emph{is a convex subset of the normalized right-continuous
functions of bounded variations.} \emph{Also here} $DT\left(  F;G\right)
=\int_{a}^{b}u_{F}d\left(  G-F\right)  $, \emph{for two distributions}
$F,G\in\mathcal{D}\left[  a,b\right]  $ \emph{and}
\[
\left(  \mathcal{D-D}\right)  ^{\perp}=span\left\{  1_{\left[  a,b\right]
}\right\}
\]
\hfill$\blacktriangle$
\end{example}

\begin{example}
\label{integral} \emph{Let }$S$\emph{ be a metric space and }$A$\emph{ a
convex subset of} $\mathbb{R}^{n}$. \emph{Denote by }$C$\emph{ the convex set
of all norm-bounded (i.e., }$\sup_{s\in S}\left\Vert u(s)\right\Vert
_{n}<\infty$\emph{) and continuous maps} $u:S\rightarrow A$\emph{. Given }%
$\mu\in\Delta\left(  S\right)  $ \emph{and }$F:S\times A\rightarrow\mathbb{R}%
$, \emph{define }$I:C\rightarrow\mathbb{R}$ \emph{by}
\[
I(u)=\int_{S}F(s,u(s))\mathrm{d}\mu(s)
\]
\emph{Under the following conditions:}

\begin{enumerate}
\item[\emph{(i)}] $F(\cdot,x)$ \emph{is Borel measurable for every} $x\in A$;

\item[\emph{(ii)}] $F(s,\cdot)$ \emph{is wa-differentiable for every} $s\in
S$, \emph{with differential }$DF\left(  s,x;\cdot\right)  $;

\item[\emph{(iii)}] \emph{there exists} $\kappa:S\rightarrow\mathbb{\mathbb{R}%
}$\emph{, with} $\int\kappa\mathrm{d}\mu<\infty$\emph{, such that}%
\[
\left\vert F\left(  \cdot,x\right)  -F\left(  \cdot,y\right)  \right\vert
\leq\kappa\left(  \cdot\right)  \left\Vert x-y\right\Vert _{n}\qquad\forall
x,y\in A
\]

\end{enumerate}

\noindent$I$\emph{ is well defined and wa-differentiable at each }$u\in
C$\emph{, with}%
\[
DI(u;g)=\int_{S}DF(s,u(s);g\left(  s\right)  )d\mu(s)\qquad\forall g\in C
\]
\emph{ Indeed, by (iii) we have }
\[
\frac{1}{t}|F(s,(1-t)u(s)+tg(s))-F(s,u(s))|\text{ }\leq\kappa\left(  s\right)
\left\Vert g(s)-u(s)\right\Vert _{n}\qquad\forall s\in S
\]
\emph{The desired result thus follows from the Dominated Convergence
Theorem.}\hfill$\blacktriangle$
\end{example}

Inspired by an example in Phelps \cite{phelps}, next we present a
wa-differentiable function which is not a-differentiable.

\begin{example}
\label{ex:hilb}\emph{Let} $C=\left\{  \left\{  x_{n}\right\}  \in
\mathbb{R}^{\mathbb{N}}:\text{ }\left\vert x_{n}\right\vert \leq a_{n}\text{
}\forall n\right\}  $, \emph{where }$\left\{  a_{n}\right\}  $\emph{ is an
assigned scalar sequence such that }$0<a_{n}<1$\emph{ and }$\sum_{n}%
a_{n}<\infty$\emph{. Clearly the set }$C$\emph{ is a norm-bounded and convex
closed subset of} \emph{the Hilbert space} $\ell^{2}$\emph{. The convex
function} $f:C\rightarrow\mathbb{R}$ \emph{defined by}%
\[
f\left(  x\right)  =\left(  \sum_{n}x_{n}\right)  ^{2}%
\]
\emph{ is wa-differentiable at each point }$x\in C$\emph{, and}%
\[
Df\left(  x;y\right)  =2\sum_{n}x_{n}\cdot\sum_{n}y_{n}-2f\left(  x\right)
\qquad\forall y\in C
\]
\emph{Clearly, }$Df\left(  x;\cdot\right)  $ \emph{is affine on }$C$\emph{. But, }$f$\emph{ is not a-differentiable at any point }$x\in C$ \emph{since there is no element }$u=\left(  u_{n}\right)  $ $\in$ $\ell^{2}$
\emph{such that }$\left(  u,y\right)  =\sum_{n}y_{n}$ \emph{for all} $y\in
C.$\emph{ Indeed, for each integer }$m$ \emph{consider the point}
$y_{m}=(0,0,..,a_{m},....)\in C$.\emph{ It follows that }$u_{m}a_{m}=\left(
u,y_{m}\right)  =a_{m}$, \emph{namely, }$u_{m}=1$\emph{ for all }$n$\emph{.
But then} $u\notin\ell^{2}$\emph{.}

\emph{The map }$y\mapsto\sum_{n}y_{n}$\emph{ is continuous. Indeed, if the
sequence }$\left\{  y_{n}^{m}\right\}  $\emph{ converges to }$\left\{
y_{n}\right\}  $\emph{ as }$m\rightarrow\infty$\emph{, then it converges
pointwise, i.e., }$y_{n}^{m}\rightarrow y_{n}$\emph{ for every }$n$\emph{. By
the Dominated Convergence Theorem, }$\sum_{n}y_{n}^{m}\rightarrow\sum_{n}%
y_{n}$\emph{. Hence, }$Df\left(  x;\cdot\right)  $ \emph{is weakly
continuous.} \emph{Finally, observe that }$\operatorname*{cor}C$\emph{ is
empty, something not surprising in light of Proposition} \emph{\ref{prop:xet}%
}.\hfill$\blacktriangle$
\end{example}

A bifunction $B:C\times C\rightarrow\mathbb{R}$ is a \emph{biaffine form }when
it is affine in each argument. It is called \emph{extendable} when it can be
extended to a bilinear form on $X\times X$. An example is the biaffine form
$B:C\times C\rightarrow\mathbb{R}$ given by%

\[
B\left(  x,y\right)  =\left(  \sum_{n}x_{n}\right)  \cdot\left(  \sum_{n}%
y_{n}\right)
\]
where $C$ is the convex subset of $\ell^{2}$ described in Example
\ref{ex:hilb}. As already seen, this form is not extendable.

A biaffine form $B$ is \emph{symmetric }when $B\left(  x,y\right)  =B\left(
y,x\right)  $ for all $x,y\in C$. Its symmetrization is
\[
B_{S}\left(  x,y\right)  =\frac{1}{2}\left[  B\left(  x,y\right)  +B\left(
y,x\right)  \right]
\]
Associated with a biaffine form $B$ there is the (\emph{affine})
\emph{quadratic form} $Q:C\rightarrow\mathbb{R}$ given by%
\[
Q\left(  x\right)  =B(x,x)=B_{S}\left(  x,x\right)
\]
The wa-differentiability of biaffine forms as well as of quadratic forms are
clearly given by%
\begin{equation}
DB\left(  x_{1},x_{2};y_{1},y_{2}\right)  =B\left(  x_{1},y_{2}\right)
+B\left(  y_{1},x_{2}\right)  -2B\left(  x_{1},x_{2}\right)  \label{eq:bilin}%
\end{equation}
and%
\begin{equation}
DQ\left(  x;y\right)  =2B_{S}(x,y)-2Q\left(  x\right)  \label{eq:quad}%
\end{equation}
The wa-derivative $DQ\left(  x;\cdot\right)  $ is extendable when the biaffine
form $B$ is extendable and we get the gradient $\nabla_{a}Q\left(  x\right)
=2B_{S}(x,\cdot)$.

\begin{example}
\label{ex:alg}\emph{The }Mann-Whitney \emph{biaffine form }$B:\Delta\left(
\mathbb{R}\right)  \times\Delta\left(  \mathbb{R}\right)  \rightarrow
\mathbb{R}$ \emph{is}%
\[
B\left(  \mu,\lambda\right)  =\int F_{\mu}\left(  t\right)  \mathrm{d}%
F_{\lambda}(t)=\int_{\mathbb{R}}F_{\mu}\mathrm{d}\lambda\text{ }%
\]
\emph{Here }$F_{\mu}$\emph{ is the cumulative distribution function associated
with }$\mu\in\Delta\left(  \mathbb{R}\right)  $ \emph{and }$\int F_{\mu
}dF_{\lambda}$ $\emph{is}$ $\emph{the}$ $\emph{Lebesgue}$\emph{-}%
$\emph{Stieltjes}$\emph{ integral. This biaffine form} \emph{is used in
statistics (see \cite{gill}). By (\ref{eq:bilin}), }
\[
DB\left(  \mu,\lambda;\mu_{1},\lambda_{1}\right)  =\int\left(  F_{\mu_{1}%
}-F_{\mu}\right)  \mathrm{d}\lambda+\int F_{\mu}\mathrm{d}\left(  \lambda
_{1}-\lambda\right)  =\int F_{\mu_{1}}\mathrm{d}\lambda+\int F_{\mu}%
\mathrm{d}\lambda_{1}+\gamma
\]
\emph{where }$\gamma$ \emph{is a scalar independent of }$\mu_{1}$ \emph{and}
$\lambda_{1}$\emph{. This wa-differential is not always extendable.} \emph{It
is, however, extendable when }$F_{\mu}$ \emph{and} $F_{\lambda}$\emph{ are
continuous with bounded support, say contained in an interval }$\left[
a,b\right]  $\emph{. Indeed,}%
\[
\int F_{\mu_{1}}\mathrm{d}\lambda+\int F_{\mu}\mathrm{d}\lambda_{1}=\int
F_{\mu_{1}}\mathrm{d}F_{\lambda}+\int F_{\mu}\mathrm{d}F_{\lambda_{1}}%
=\int_{a}^{b}F_{\mu_{1}}\mathrm{d}F_{\lambda}+\int F_{\mu}\mathrm{d}%
F_{\lambda_{1}}=1-\int F_{\lambda}\mathrm{d}F_{\mu_{1}}+\int F_{\mu}%
\mathrm{d}F_{\lambda_{1}}%
\]
\emph{where, by the continuity of the distribution functions, the integral
}$\int_{a}^{b}F_{\mu_{1}}dF_{\lambda}$\emph{ is Riemann-Stieltjes}%
.\footnote{In the last equality we used integration by parts (see, e.g.,
Theorem 14.10 of \cite{caro}).} \emph{Therefore, }$B$\emph{ is
a-differentiable at }$\left(  \mu,\lambda\right)  $,\emph{ with gradient}
\[
\nabla_{a}B\left(  \mu,\lambda\right)  =\left(  -F_{\lambda},F_{\mu}\right)
\in C_{b}\left(  \mathbb{R}\right)  \times C_{b}\left(  \mathbb{R}\right)
\]
\hfill$\blacktriangle$
\end{example}

\subsection{Mean value theorem}

Let $x,y\in C$ and, for each $t\in\left[  0,1\right]  $, set $x_{t}=\left(
1-t\right)  x+ty.$ Each function $f:C\rightarrow\mathbb{R}$ has an auxiliary
scalar function $\varphi_{x,y}:\left[  0,1\right]  \rightarrow\mathbb{R}$
defined by $\varphi_{x,y}\left(  t\right)  =f\left(  x_{t}\right)  $.

\begin{definition}
\label{def:hem}A function $f:C\rightarrow\mathbb{R}$ is
\emph{hemidifferentiable} if, for all $x,y\in C$, its auxiliary function
$\varphi_{x,y}:\left[  0,1\right]  \rightarrow\mathbb{R}$ is differentiable on
$\left[  0,1\right]  $.\footnote{Differentiable on $\left[  0,1\right]  $
means right-differentiable at $t=0$ and left-differentiable at $t=1$.}
\end{definition}

We begin with a non-trivial lemma. To ease on notation, when no confusion may
arise we often omit subscripts and just write $\varphi$ in place of
$\varphi_{x,y}$.

\begin{lemma}
\label{lmm:ccc}A wa-differentiable\footnote{A function is \textquotedblleft
wa-differentiable\textquotedblright\ (\textquotedblleft
a-differentiable\textquotedblright) when wa-differentiable (a-differentiable)
at all points of its domain.} $f:C\rightarrow\mathbb{R}$ is
hemidifferentiable, with%
\begin{equation}
\varphi^{\prime}\left(  t\right)  =\frac{1}{1-t}Df\left(  x_{t};y\right)
=-\frac{1}{t}Df\left(  x_{t};x\right)  \qquad\forall t\in\left(  0,1\right)
\label{eq:ghl}%
\end{equation}
and%
\[
\varphi_{+}^{\prime}\left(  0\right)  =Df\left(  x;y\right)  \quad
\text{;}\quad\varphi_{-}^{\prime}\left(  1\right)  =-Df\left(  y;x\right)
\]

\end{lemma}

As $Df\left(  x_{t};\cdot\right)  $ is affine, (\ref{eq:ghl}) is equivalent
to
\begin{equation}
\varphi^{\prime}\left(  t\right)  =\frac{1}{\tau-t}Df\left(  x_{t};x_{\tau
}\right)  \label{eq:gen}%
\end{equation}
for all $\left(  t,\tau\right)  \in\left(  0,1\right)  \times\left[
0,1\right]  $ with $\tau\neq t$. For $\tau\in\left\{  0,1\right\}  $ we obtain
(\ref{eq:ghl}).

When $f$ is a-differentiable on $C$, from (\ref{eq:hju}) it follows
immediately that, for each $t\in\left(  0,1\right)  $,\footnote{Recall that
$\nabla_{a}f\left(  x_{t}\right)  $ is a representative of the equivalence
class $\left[  \nabla_{a}f\left(  x_{t}\right)  \right]  $, so (\ref{eq:typ})
means $\varphi^{\prime}\left(  t\right)  =\left\langle y-x,x^{\ast
}\right\rangle $ for any $x^{\ast}\in\left[  \nabla_{a}f\left(  x_{t}\right)
\right]  $.}%
\begin{equation}
\varphi^{\prime}\left(  t\right)  =\frac{1}{1-t}Df\left(  x_{t};y\right)
=\left\langle y-x,\nabla_{a}f\left(  x_{t}\right)  \right\rangle .
\label{eq:typ}%
\end{equation}

\noindent\textbf{Proof} For each $t\in\left[  0,1\right]  $ we have the
following obvious algebraic relations
\[
y-x_{t}=\left(  1-t\right)  \left(  y-x\right)  \quad\text{and\quad}%
x-x_{t}=t\left(  x-y\right)
\]
Now fix $x,y\in C$ and $t\in\left[  0,1\right)  $. The limit%
\[
\lim_{h\downarrow0}\frac{\varphi\left(  t+\left(  1-t\right)  h\right)
-\varphi\left(  t\right)  }{h}%
\]
exists. Indeed,%
\[
\frac{\varphi\left(  t+\left(  1-t\right)  h\right)  -\varphi\left(  t\right)
}{h}=\frac{f\left(  x_{t}+h\left(  1-t\right)  \left(  y-x\right)  \right)
-f\left(  x_{t}\right)  }{h}=\frac{f\left(  x_{t}+h\left(  y-x_{t}\right)
\right)  -f\left(  x_{t}\right)  }{h}%
\]
Hence,%
\[
\lim_{h\downarrow0}\frac{\varphi\left(  t+\left(  1-t\right)  h\right)
-\varphi\left(  t\right)  }{h}=Df\left(  x_{t};y\right)
\]
On the other hand,
\[
\lim_{h\downarrow0}\frac{\varphi\left(  t+\left(  1-t\right)  h\right)
-\varphi\left(  t\right)  }{h}=\left(  1-t\right)  \lim_{h\downarrow0}%
\frac{\varphi\left(  t+\left(  1-t\right)  h\right)  -\varphi\left(  t\right)
}{\left(  1-t\right)  h}=\left(  1-t\right)  \varphi_{+}^{\prime}\left(
t\right)
\]
Hence,%
\[
\varphi_{+}^{\prime}\left(  t\right)  =\frac{1}{1-t}Df\left(  x_{t};y\right)
\]

Use the same method for the left derivative $\varphi_{-}^{\prime}$.
Specifically, begin with the limit
\begin{align*}
\lim_{h\uparrow0}\frac{\varphi\left(  t+th\right)  -\varphi\left(  t\right)
}{h}  &  =-\lim_{k\downarrow0}\frac{\varphi\left(  t-tk\right)  -\varphi
\left(  t\right)  }{k}=-\lim_{k\downarrow0}\frac{f\left(  x_{t}+kt\left(
x-y\right)  \right)  -f\left(  x_{t}\right)  }{k}\\
&  =-\lim_{k\downarrow0}\frac{f\left(  x_{t}+k\left(  x-x_{t}\right)  \right)
-f\left(  x_{t}\right)  }{k}=-Df\left(  x_{t};x\right)
\end{align*}
At the same time, we have%
\[
\varphi_{-}^{\prime}\left(  t\right)  =\lim_{h\uparrow0}\frac{\varphi\left(
t+th\right)  -\varphi\left(  t\right)  }{th}=\frac{1}{t}\lim_{h\uparrow0}%
\frac{\varphi\left(  t+th\right)  -\varphi\left(  t\right)  }{h}=-\frac{1}%
{t}Df\left(  x_{t};x\right)
\]
Since $Df\left(  x_{t};\cdot\right)  $ is affine, we obtain
\[
0=Df\left(  x_{t};x_{t}\right)  =Df\left(  x_{t};\left(  1-t\right)
x+ty\right)  =\left(  1-t\right)  Df\left(  x_{t};x\right)  +tDf\left(
x_{t};y\right)
\]
which implies%
\[
\frac{1}{1-t}Df\left(  x_{t};y\right)  =-\frac{1}{t}Df\left(  x_{t};x\right)
=\varphi_{-}^{\prime}\left(  t\right)  =\varphi_{+}^{\prime}\left(  t\right)
\]
Hence, $\varphi_{+}^{\prime}=\varphi_{-}^{\prime}$ on $\left(  0,1\right)  $
and the proof is complete.\hfill$\blacksquare$

\bigskip

It is apparent from the proof of this lemma that when $f$ is not
wa-differentiable, but has one-sided directional derivatives, we can still
infer that
\[
\varphi_{+}^{\prime}\left(  t\right)  =\frac{1}{1-t}Df\left(  x_{t};y\right)
\text{\quad and\quad}\varphi_{-}^{\prime}\left(  t\right)  =-\frac{1}%
{t}Df\left(  x_{t};x\right)
\]

A remarkable consequence of the last lemma is the following mean value theorem.

\begin{theorem}
[Mean value theorem]\label{cor:mean} Let $f:C\rightarrow\mathbb{R}$ be
wa-differentiable. For each $x,y\in C$ there exists $t\in\left(  0,1\right)  $
such that%
\begin{equation}
f\left(  y\right)  -f\left(  x\right)  =\frac{1}{1-t}Df\left(  x_{t};y\right)
\label{eq:mv}%
\end{equation}

\end{theorem}

When $f$ is a-differentiable, by (\ref{eq:typ}) we get%
\begin{equation}
f\left(  y\right)  -f\left(  x\right)  =\left\langle y-x,\nabla_{a}f\left(
x_{t}\right)  \right\rangle \label{eq:AAG}%
\end{equation}

\noindent\textbf{Proof} Let $x,y\in C$. The auxiliary function $\varphi\left(
t\right)  =f\left(  x_{t}\right)  $ is continuous on $\left[  0,1\right]  $
and differentiable on $\left(  0,1\right)  $. By the basic mean value theorem,
there exists $t\in\left(  0,1\right)  $ such that%
\[
\varphi^{\prime}\left(  t\right)  =\frac{\varphi\left(  1\right)
-\varphi\left(  0\right)  }{1-0}=\varphi\left(  1\right)  -\varphi\left(
0\right)
\]
By Lemma \ref{lmm:ccc},%
\[
\frac{1}{1-t}Df\left(  x_{t};y\right)  =f\left(  y\right)  -f\left(  x\right)
\]
as desired.\hfill$\blacksquare$

\subsection{Affine calculus}

To develop an effective affine calculus, introducing a slightly stronger
notion of differentiability can be useful.

\begin{definition}
A wa-differentiable function $f:C\rightarrow\mathbb{R}$ is (\emph{radially}%
)\emph{ continuously} \emph{wa-differentiable} if, for each $x,y\in C$, the
functions
\begin{equation}
t\mapsto Df\left(  x_{t};y\right)  \text{\quad and\quad}t\mapsto Df\left(
x_{t};x\right)  \label{eq:lolo}%
\end{equation}
are both continuous on $\left[  0,1\right]  $.
\end{definition}

For instance, it is easy to check that all biaffine forms as well as all
quadratic functionals are radially continuously wa-differentiable.

By (\ref{eq:ghl}), we have%
\[
Df\left(  x_{t};x\right)  =-\frac{t}{1-t}Df\left(  x_{t};y\right)
\]
As a consequence, in (\ref{eq:lolo}) the continuity of $t\mapsto Df\left(
x_{t};y\right)  $ implies that of $t\mapsto Df\left(  x_{t};x\right)  $ at all
$t\in\left[  0,1\right)  $. Only at $t=1$, the continuity of $t\mapsto
Df\left(  x_{t};x\right)  $ is a genuine assumption.

\begin{proposition}
\label{prop:funda}If $f:C\rightarrow\mathbb{R}$ is continuously
wa-differentiable, then%
\[
f\left(  y\right)  -f\left(  x\right)  =\int_{0}^{1}\frac{1}{1-t}Df\left(
x_{t};y\right)  \mathrm{d}t
\]
for all $x,y\in C$.
\end{proposition}

When $f$ is a-differentiable, we get%
\[
f\left(  y\right)  -f\left(  x\right)  =\int_{0}^{1}\left\langle
y-x,\nabla_{a}f\left(  x_{t}\right)  \right\rangle \mathrm{d}t
\]

\noindent\textbf{Proof} Consider the auxiliary function $\varphi\left(
t\right)  =f\left(  x_{t}\right)  $. By Lemma \ref{lmm:ccc}, $\varphi^{\prime
}\left(  t\right)  $ exists for all $t\in\left(  0,1\right)  $. Moreover,
$\varphi^{\prime}$ is continuous on $\left(  0,1\right)  $ by the relation
$\varphi^{\prime}\left(  t\right)  =\left(  1-t\right)  ^{-1}Df\left(
x_{t};y\right)  $ and by the continuity of the first map in (\ref{eq:lolo}).
On the other hand, by the continuity of first map in (\ref{eq:lolo})
\[
\lim_{t\rightarrow0^{+}}\varphi^{\prime}\left(  t\right)  =\lim_{t\rightarrow
0^{+}}\left(  1-t\right)  ^{-1}Df\left(  x_{t};y\right)  =Df\left(
x;y\right)  =\varphi_{+}^{\prime}\left(  0\right)
\]
Analogously, $\lim_{t\rightarrow1-}\varphi^{\prime}\left(  t\right)  =$
$\varphi_{-}^{\prime}\left(  1\right)  $. As a result, $\varphi^{\prime}$ is
continuous on $\left(  0,1\right)  $ and bounded on $\left[  0,1\right]  $.
Hence,%
\[
f\left(  y\right)  -f\left(  x\right)  =\varphi\left(  1\right)
-\varphi\left(  0\right)  =\int_{0}^{1}\varphi^{\prime}\left(  t\right)
\mathrm{d}t=\int_{0}^{1}\frac{1}{1-t}Df\left(  x_{t};y\right)  \mathrm{d}t
\]
as desired.\hfill$\blacksquare$

\begin{example}
\emph{According to} \emph{Example \ref{ex:lop}, when }$T:\Delta\left(
Y\right)  \rightarrow\mathbb{R}$ \emph{is} \emph{a-differentiable,} \emph{by
the mean value theorem (Theorem \ref{cor:mean}) there is a scalar }%
$t\in\left(  0,1\right)  $\emph{ such that}%
\[
T\left(  \lambda\right)  -T\left(  \mu\right)  =\frac{1}{1-t}\int_{Y}%
u_{\mu_{t}}\mathrm{d}\left(  \lambda-\mu\right)
\]
\emph{In addition,} \emph{Proposition \ref{prop:funda} claims that it also
holds}%
\[
T\left(  \lambda\right)  -T\left(  \mu\right)  =\int_{0}^{1}\left(  \int
_{Y}u_{\mu_{t}}\mathrm{d}\left(  \lambda-\mu\right)  \right)  \mathrm{d}t
\]
\emph{when }$T$ \emph{is radially continuously a-differentiable.}%
\hfill$\blacktriangle$
\end{example}

To characterize convexity for wa-differentiable functions we need the
following monotonicity notion. Notice that some of the following properties
remain valid even for functions which have affine directional
derivatives\footnote{Recall that $f$ has affine directional derivative at $x$,
if $Df\left(  x;y\right)  $ exists  and is finite for every $y\in C$ (see Definition
\ref{def:dir}).} and not necessarily wa-differentiable.

\begin{definition}
\label{def:mono}Let $f:C\rightarrow\mathbb{R}$ have an affine directional
derivative at every $x\in C$. The function $f$ is said to have a
\emph{monotone} \emph{affine} \emph{directional derivative }if%
\begin{equation}
Df\left(  x;y\right)  +Df\left(  y;x\right)  \leq0\text{ } \label{eq:mon}%
\end{equation}
holds for all $x,y\in C$.
\end{definition}

This property has a sharper form for a-differentiable functions because
equation (\ref{eq:mon}) becomes the more familiar monotonicity condition
\[
\left\langle y-x,\nabla_{a}f\left(  y\right)  -\nabla_{a}f\left(  x\right)
\right\rangle \geq0
\]
With this, we can characterize convexity through the affine directional derivative.

\begin{proposition}
\label{prop:AAA}Consider the following properties:

\begin{enumerate}
\item[(i)] $f$ is convex;

\item[(ii)] for all $x,y\in C$,%
\begin{equation}
f\left(  y\right)  \geq f\left(  x\right)  +Df\left(  x;y\right)
\label{eq:cob}%
\end{equation}

\item[(iii)] the directional affine derivative is monotone.\medskip

It holds $\left(  i\right)  \Longrightarrow\left(  ii\right)  \Longrightarrow
(iii)$ and the three properties are equivalent when $f$ is wa-differentiable.
\end{enumerate}
\end{proposition}

The convexity of a quadratic functional $Q:C\rightarrow\mathbb{R}$ is a simple
illustration of this result. In view of (\ref{eq:mon}), we obtain the
condition
\[
Q\left(  x\right)  +Q\left(  y\right)  -2B_{S}\left(  x,y\right)
\geq0\text{\qquad}\forall x,y\in C
\]
When the biaffine form is extendable, this inequality becomes the familiar
condition $Q\left(  x-y\right)  \geq0$ of semidefinite positivity for
quadratic functionals.

\bigskip

\noindent\textbf{Proof} We omit the proofs of (i)-(ii)-(iii), as they are
quite standard. . Let us check that (iii) implies (i) for wa-differentiable
functions. Consider the auxiliary function $\varphi\left(  t\right)  =f\left(
x_{t}\right)  $ on $\left[  0,1\right]  $ for two arbitrary points $x,y\in C$.
Let
\[
x_{t}=\left(  1-t\right)  x+ty\quad\text{and\quad}x_{\tau}=\left(
1-\tau\right)  x+\tau y
\]
with $t,\tau\in\left(  0,1\right)  $. In view of (\ref{eq:gen})%
\begin{align*}
0  &  \geq Df\left(  x_{t};x_{\tau}\right)  +Df\left(  x_{\tau};x_{t}\right)
=\left(  \tau-t\right)  \varphi^{\prime}\left(  t\right)  +\left(
t-\tau\right)  \varphi^{\prime}\left(  \tau\right) \\
&  =\left(  \tau-t\right)  \left(  \varphi^{\prime}\left(  t\right)
-\varphi^{\prime}\left(  \tau\right)  \right)
\end{align*}
The first derivative $\varphi^{\prime}$ is thus nondecreasing, so $\varphi$ is
convex on $\left(  0,1\right)  .$ By continuity of $\varphi$ in $\left[
0,1\right]  $, $\varphi$ is convex on the closed interval as well. Since this
is true for any pair of points $x,y\in C$, we conclude that $f$ is
convex.\hfill$\blacksquare$

\bigskip

Thanks to the previous results we can also relate gradients and
subdifferentials for convex wa-differentiable functionals. The
\emph{subdifferential} of a convex function $f:C\rightarrow\mathbb{R}$ at
$x\in C$ is the set
\[
\partial f(x)=\{x^{\ast}\in X^{\ast}:\forall y\in C,\text{ }f(y)\geq f\left(
x\right)  +\left\langle y-x,x^{\ast}\right\rangle \}
\]
while the (\emph{negative conical})\emph{ polar }$A^{-}\subseteq X^{\ast}$ of
a set $A$ in $X$ is given by%
\[
A^{-}=\left\{  x^{\ast}\in X^{\ast}:\forall x\in A,\text{ }\left\langle
x,x^{\ast}\right\rangle \leq0\right\}
\]
For a convex set $C$, the \emph{normal cone }at a point\emph{ }$x\in C$ is
defined as $N_{C}\left(  x\right)  =\left(  C-x\right)  ^{-}$.

\begin{proposition}
\label{prop:huhu}Let $f:C\rightarrow\mathbb{R}$ be a-differentiable and
convex. For each $x\in C$,%
\begin{equation}
\partial f\left(  x\right)  =x^{\ast}+N_{C}\left(  x\right)  \qquad\forall
x^{\ast}\in\left[  \nabla_{a}f\left(  x\right)  \right]  \label{eq:lola}%
\end{equation}
Moreover, if $x\in\operatorname*{ri}C$,
\[
\partial f\left(  x\right)  =\left[  \nabla_{a}f\left(  x\right)  \right]
\]

\end{proposition}

At non-internal points of $C$, the subdifferential of a convex function can
thus be strictly larger than its a-differential because the vector space
$\left(  C-x\right)  ^{\perp}$ can be strictly included in the normal cone
$N_{C}\left(  x\right)  $.

\bigskip

\noindent\textbf{Proof} Let $x^{\ast}\in\left[  \nabla_{a}f\left(  x\right)
\right]  $. By (\ref{eq:cob}),
\[
f\left(  y\right)  \geq f\left(  x\right)  +Df\left(  x;y\right)  =f\left(
x\right)  +\left\langle y-x,x^{\ast}\right\rangle \qquad\forall y\in C
\]
For an arbitrary element $p^{\ast}$ of $\left(  C-x\right)  ^{-}$, we have
$\left\langle y-x,p^{\ast}\right\rangle \leq0$. Consequently,%
\[
f\left(  y\right)  \geq f\left(  x\right)  +\left\langle y-x,x^{\ast
}\right\rangle \geq f\left(  x\right)  +\left\langle y-x,x^{\ast}+p^{\ast
}\right\rangle
\]
Hence, $x^{\ast}+p^{\ast}$ is a subdifferential. We thus proved the inclusion
$x^{\ast}+\left(  C-x\right)  ^{-}\subseteq\partial f\left(  x\right)  $.

Conversely, let $p^{\ast}\in\partial f\left(  x\right)  $. Then, $f\left(
y\right)  -f\left(  x\right)  \geq\left\langle y-x,p^{\ast}\right\rangle $ and
so%
\[
f\left(  x_{t}\right)  -f\left(  x\right)  \geq t\left\langle y-x,p^{\ast
}\right\rangle
\]
for all $y\in C$ and $t\in\left[  0,1\right]  $. Dividing by $t$ and letting
$t\rightarrow0$, we get $Df\left(  x;y\right)  \geq\left\langle y-x,p^{\ast
}\right\rangle $. If $x^{\ast}$ is a gradient,
\[
\left\langle y-x,x^{\ast}\right\rangle =Df\left(  x;y\right)  \geq\left\langle
y-x,p^{\ast}\right\rangle \Longrightarrow0\geq\left\langle y-x,p^{\ast
}-x^{\ast}\right\rangle
\]
That is, $p^{\ast}\in x^{\ast}+\left(  C-x\right)  ^{-}$. This proves the
converse inclusion $\partial f\left(  x\right)  \subseteq x^{\ast}+\left(
C-x\right)  ^{-}$.

In sum, $\partial f\left(  x\right)  =x^{\ast}+N_{C}\left(  x\right)  $. The
equality $\partial f\left(  x\right)  =\left[  \nabla_{a}f\left(  x\right)
\right]  $ is a consequence of the fact that $\left(  C-x\right)  ^{-}=\left(
C-x\right)  ^{\perp}$ when $x\in\operatorname*{ri}C$.\hfill$\blacksquare$

\bigskip

This proposition implies that subdifferentials are not empty for
a-differentiable convex functions. Thus, a-differentiable convex functions are
weakly lower semicontinuous.

We also report the following sum rule for subdifferentials.

\begin{proposition}
If two convex functions $f,g:C\rightarrow\mathbb{R}$ are a-differentiable at
$x\in C$, then%
\[
\partial\left(  f+g\right)  \left(  x\right)  =\partial f\left(  x\right)
+\partial g\left(  x\right)
\]

\end{proposition}

\noindent\textbf{Proof} It follows from the relation $\left(  C-x\right)
^{-}+\left(  C-x\right)  ^{-}=\left(  C-x\right)  ^{-}$. Hence, if $x^{\ast
}\in\left[  \nabla_{a}f\left(  x\right)  \right]  $ and $y^{\ast}\in\left[
\nabla_{a}g\left(  x\right)  \right]  $, we have%
\[
\partial\left(  f+g\right)  \left(  x\right)  =x^{\ast}+y^{\ast}+\left(
C-x\right)  ^{-}=x^{\ast}+\left(  C-x\right)  ^{-}+y^{\ast}+\left(
C-x\right)  ^{-}=\partial f\left(  x\right)  +\partial g\left(  x\right)
\]
as desired.\hfill$\blacksquare$

\bigskip
\begin{sloppypar}
We close by considering the characterization of \emph{quasiconvex functionals }through their wa-differentials.\end{sloppypar} Recall that $f:C\rightarrow\mathbb{R}$ is quasiconvex if
\[
f\left(  tx+\left(  1-t\right)  y\right)  \leq\max\left\{  f\left(  x\right)
,f\left(  y\right)  \right\}
\]
holds for all $x,y\in C$ and for all $t\in\left(  0,1\right)  $. Relation
(\ref{eq:diquac}) below is the wa-differentiable extension of the
Arrow-Enthoven \cite{Arrow} result.

\begin{proposition}
\label{prop:quasiconvex}A wa-differentiable $f:C\rightarrow\mathbb{R}$ is
quasiconvex if and only if
\begin{equation}
f\left(  y\right)  \leq f\left(  x\right)  \Longrightarrow Df(x;y)\leq0
\label{eq:diquac}%
\end{equation}
for all $x,y\in C$.
\end{proposition}

\noindent\textbf{Proof} Let $f:C\rightarrow\mathbb{R}$ be wa-differentiable.
The function $f$ is quasiconvex if and only if, for all $x,y\in C$, its
restrictions $\varphi_{x,y}\left(  t\right)  =f\left(  x_{t}\right)  $ are
quasiconvex. By Lemma \ref{lmm:ccc}, $f$ is hemidifferentiable. The
characterization of one-dimensional quasiconvex and differentiable functions
is
\begin{equation}
\varphi_{x,y}\left(  t_{2}\right)  \leq\varphi_{x,y}\left(  t_{1}\right)
\Longrightarrow\varphi_{x,y}^{\prime}\left(  t_{1}\right)  (t_{2}-t_{1})\leq0
\label{eq:difqual}%
\end{equation}
for all $t_{1},t_{2}\in\left[  0,1\right]  $ (See for instance \cite[p.90]%
{crouzeix}). In view of (\ref{eq:gen}), it becomes
\[
f\left(  x_{t_{2}}\right)  \leq f\left(  x_{t_{1}}\right)  \Longrightarrow
\frac{1}{t_{2}-t_{1}}Df(x_{t_{1}};x_{t_{2}})(t_{2}-t_{1})\leq0
\]
for all $t_{1},t_{2}\in\left[  0,1\right]  $. By setting $t_{1}=0$ and
$t_{2}=1$ the desired result follows.\hfill$\blacksquare$

\bigskip

Proposition \ref{prop:quasiconvex} is not at all novel. There are already
several extensions of the Arrow-Enthoven characterization in the literature.
Diewert \cite[Corollary 6]{diewert} characterized quasiconvex functions for
radially lsc functions via Dini-derivatives. Without\ any further assumption
Lara \cite[Prop. 3.1]{lara} gave a characterization via global directional derivatives.

What is most interesting is that for our class of functions, we can formulate
a good definition of pseudoconvexity, extending Mangasarian's
\cite{mangasarian} result.

Recall the following notion of subdifferential useful in quasiconvex analysis
(see more details in \cite{green}). The \emph{Greenberg-Pierskalla
subdifferential} of a function $g:C\rightarrow\mathbb{R}$ at $x\in C$ is%
\[
\partial_{GP}g\left(  x\right)  =\left\{  x^{\ast}\in X^{\ast}:y\in C\text{
and }\left\langle y-x,x^{\ast}\right\rangle \geq0\Longrightarrow g\left(
y\right)  \geq g\left(  x\right)  \right\}  \text{.}%
\]

\begin{definition}
\label{def:pseudo}A wa-differentiable function $f:C\rightarrow\mathbb{R}$ is
said to be \emph{pseudoconvex} if%
\begin{equation}
f\left(  y\right)  <f\left(  x\right)  \Longrightarrow Df(x;y)<0
\label{eq:pseu}%
\end{equation}
for all $x,y\in C.$ When $f$ is a-differentiable, equation (\ref{eq:pseu}) is
equivalent to $\nabla_{a}f\left(  x\right)  \subseteq\partial_{GP}f\left(
x\right)  $ for every $x\in C$.
\end{definition}

Observe that (\ref{eq:pseu}) is equivalent to $Df(x;y)\geq0\Longrightarrow
f\left(  y\right)  \geq f\left(  x\right)  $ and $Df(x;y)\geq0$ means
$\left\langle \nabla_{a}f\left(  x\right)  ,y-x\right\rangle \geq0$, whenever
$f$ is a-differentiable. Therefore the inclusion $\nabla_{a}f\left(  x\right)
\subseteq\partial_{GP}f\left(  x\right)  $ is true. The next results justify
our Definition \ref{def:pseudo}.

\begin{proposition}
Let $f:C\rightarrow\mathbb{R}$ be wa-differentiable.

\begin{enumerate}
\item[(i)] if $f$ is convex, then $f$ is pseudoconvex;

\item[(ii)] if $f$ is pseudoconvex, $f$ is quasiconvex, provided $f$ is a-differentiable.
\end{enumerate}
\end{proposition}

\begin{proof}
Point (i) is a straightforward consequence of the characterization in
(\ref{eq:cob}) given for convex functions.

As for (ii), since $\nabla_{a}f\left(  x\right)  \subseteq\partial
_{GP}f\left(  x\right)  $, it follows that $\partial_{GP}f\left(  x\right)
\neq\emptyset$ for every $x\in C$. This implies that $f$ is quasiconvex (for a
proof we refer to \cite[Corollary 4]{green}).
\end{proof}

\section{Optimization\label{ch2}}

We begin with the somewhat familiar first-order condition for a (global) minimizer.

\begin{proposition}
\label{ottimo}Let $f:C\rightarrow\mathbb{R}$ have affine directional
derivative. If $\hat{x}\in C$ is a minimizer of $f$, then%
\begin{equation}
Df\left(  \hat{x};y\right)  \geq0\qquad\forall y\in C. \label{eq:foc-bis}%
\end{equation}
Moreover, it holds with equality when $\hat{x}\in\operatorname*{ri}C$,
provided $f$ is wa-differentiable.
\end{proposition}

Since $Df\left(  \hat{x};\hat{x}\right)  =0$, an equivalent condition to
(\ref{eq:foc-bis}) is clearly
\begin{equation}
\min_{y\in C}\text{ }Df\left(  \hat{x};y\right)  =0. \label{eq:SSS}%
\end{equation}
Observe further that the necessary condition (\ref{eq:foc-bis}) turns out to
be sufficient as well, provided $f$ is pseudoconvex.

\bigskip

\noindent\textbf{Proof} \ Condition (\ref{eq:foc-bis}) is standard and its
proof is omitted.

It remains to show that if $\hat{x}\in\operatorname*{ri}C$ and $f$ is
wa-differentiable then its differential vanishes. Fix a point $y\in C$ that
differs from $\hat{x}$. As $\hat{x}\in\operatorname*{ri}C$, there exist $x\in
C$ and $\bar{t}\in\left(  0,1\right)  $ such that $\left(  1-\bar{t}\right)
x+\bar{t}y=\hat{x}$. By Lemma \ref{lmm:ccc}, the function $\varphi\left(
t\right)  =f\left(  x_{t}\right)  $ is differentiable at $\bar{t}$ and
\[
\varphi^{\prime}\left(  \bar{t}\right)  =\frac{1}{1-\bar{t}}Df\left(
x_{\bar{t}};y\right)  =\frac{1}{1-\bar{t}}Df\left(  \hat{x};y\right)
\]
On the other hand $\varphi$ has a minimizer at the interior point $\bar{t}$.
Hence, $\varphi^{\prime}\left(  \bar{t}\right)  =0$ and so $Df\left(  \hat
{x};y\right)  =0$ for all $y\in C$.\hfill$\blacksquare$

\bigskip

Let $S$ be a topological space and $f:S\times C\rightarrow\mathbb{R}$ a
parametric objective function (here $C$ is interpreted a set of parameters).
Define the \emph{value function} $v:C\rightarrow\mathbb{R}$ by
\[
v(x)=\sup_{s\in S}f(s,x)
\]
and the associated \emph{solution correspondence} $\sigma:C\rightrightarrows
S$ by%
\[
\sigma\left(  x\right)  =\arg\max_{s\in S}f(s,x)
\]
We say that $\sigma$ is \emph{viable} when $\sigma\left(  x\right)
\neq\emptyset$ for all $x\in C$.

A piece of notation: for a fixed element $s\in S$, $Df_{s}(x;y)$ denotes the
affine directional derivative at $x\in C$ of the section $f_{s}:C\rightarrow
\mathbb{R}$ of the objective function $f$. Our first result provides an
estimate for the wa-differential of the value function.

\begin{proposition}
\label{envelope} Let $f_{s}:C\rightarrow\mathbb{R}$ be wa-differentiable for
each $s\in S$. If $v$ is wa-differentiable and $\sigma$ is viable, then
\begin{equation}
Dv(x;y)\geq Df_{s}(x;y)\qquad\forall y\in C \label{eq:inefoc}%
\end{equation}
for each $x\in C$ and each $s\in\sigma\left(  x\right)  $, with equality when
$x\in\operatorname*{ri}C$.
\end{proposition}

\noindent\textbf{Proof} Fix $x\in C$ and $s\in\sigma\left(  x\right)  $. The
function $\Phi:C\rightarrow\mathbb{R}$ given by $\Phi\left(  y\right)
=v\left(  y\right)  -f\left(  s,y\right)  $ is wa-differentiable. We have
$\Phi\geq0$ and $\Phi\left(  x\right)  =0$. Hence, $x$ is a minimizer and so
the first-order condition (\ref{eq:foc-bis}) gives the inequality
(\ref{eq:inefoc}). If $x\in\operatorname*{ri}C$ we then get the desired
equality.\hfill$\blacksquare$

\bigskip

The next result, a variant of the classic Danskin-Demyanov envelope
theorem\footnote{The Danskin's and Demyanov envelope theorem was
simultaneously and completely independently proved by both Danskin and
Demyanov in the mid-1960s (see \cite{dan} and references to
Demyanov's original papers and historical comments in Section 5.4.2 of
\cite{polak}). We thank an anonymous referee for mentioning Demyanov's work.},
provides conditions under which the affine directional derivative (not
necessarily the wa-derivative) of the value function does exist. As an
additional piece of notation, $x_{t}$ will denote the usual convex combination
$x_{t}=\left(  1-t\right)  x+ty$.

\begin{theorem}
\label{thm:danskin}Let $\sigma$ be viable. Given $x,y\in C$, assume that $f$
satisfies the following two properties:

\begin{enumerate}
\item[(i)] the affine directional derivative $Df_{s}(x;y)$ exists for all
$s\in\sigma\left(  x\right)  $;

\item[(ii)] for every sequence $\left\{  t_{n}\right\}  \subseteq\left[
0,1\right]  $ with $t_{n}\downarrow0$, there is a sequence $s_{n}\in
\sigma\left(  x_{t_{n}}\right)  \in S$, such that%
\begin{equation}
\limsup_{n\rightarrow\infty}\frac{f(s_{n},x_{t_{n}})-f(s_{n},x)}{t_{n}}%
\leq\sup_{s\in\sigma\left(  x\right)  }Df_{s}(x;y) \label{eq:FFF}%
\end{equation}

\end{enumerate}

Then, the affine directional derivative $Dv(x;y)$ exists and%
\[
Dv(x;y)=\sup_{s\in\sigma\left(  x\right)  }Df_{s}(x;y)
\]

\end{theorem}

\bigskip

\noindent\textbf{Proof} We first show that, for each $y\in C$,
\begin{equation}
\lim\inf_{t\downarrow0}\frac{v\left(  x_{t}\right)  -v\left(  x\right)  }%
{t}\geq\sup_{s\in\sigma\left(  x\right)  }Df_{s}\left(  x;y\right)
\label{eq:liminf}%
\end{equation}
Fix $y\in C$ and pick a point $s_{0}\in\sigma\left(  x\right)  $. Take any
sequence $t_{n}\in\left(  0,1\right)  $ with $t_{n}\downarrow0$. By
definition,%
\[
\frac{v\left(  x_{t_{n}}\right)  -v\left(  x\right)  }{t_{n}}\geq
\frac{f\left(  s_{0},x_{t_{n}}\right)  -f\left(  s_{0},x\right)  }{t_{n}}%
\]
Hence,
\[
\lim\inf_{n\rightarrow\infty}\frac{v\left(  x_{t_{n}}\right)  -v\left(
x\right)  }{t_{n}}\geq Df_{s_{0}}\left(  x;y\right)
\]
This is true for every sequence $t_{n}\downarrow0$. Consequently,%
\[
\lim\inf_{t\downarrow0}\frac{v\left(  x_{t}\right)  -v\left(  x\right)  }%
{t}\geq Df_{s_{0}}\left(  x;y\right)
\]
As this is true for every $s_{0}\in\sigma\left(  x\right)  $, we get
(\ref{eq:liminf}). To end the proof, let us prove that%
\begin{equation}
\lim\sup_{t\downarrow0}\frac{v\left(  x_{t}\right)  -v\left(  x\right)  }%
{t}\leq\sup_{s\in\sigma\left(  x\right)  }Df_{s}\left(  x;y\right)
\label{eq:limsup}%
\end{equation}
Given a sequence $t_{n}\downarrow0$, define $\left\{  s_{n}\right\}  $
according to assumption (ii). Then:
\begin{align*}
\frac{v(x_{t_n}) - v(x)}{t_n}
  &= \frac{f(s_n, x_{t_n}) - v(x)}{t_n} \\
  &= \frac{f(s_n, x_{t_n}) - f(s_n, x)}{t_n}
     + \frac{f(s_n, x) - v(x)}{t_n}
     \leq \frac{f(s_n, x_{t_n}) - f(s_n, x)}{t_n}
\end{align*}
In light of (ii) we have
\[
\lim\sup_{n\rightarrow\infty}\frac{v\left(  x_{t_{n}}\right)  -v\left(
x\right)  }{t_{n}}\leq\lim\sup_{n\rightarrow\infty}\frac{f\left(
s_{n},x_{t_{n}}\right)  -f\left(  s_{n},x\right)  }{t_{n}}\leq\sup_{s\in
\sigma\left(  x\right)  }Df_{s}(x;y)
\]
We get (\ref{eq:limsup}) that, in turn, yields the desired result.\hfill
$\blacksquare$

\bigskip

In the next example, both condition (ii) and the conclusion of the theorem are violated.
\begin{example}
\emph{Let}
\[
C=[0,1]\mathrm{ \ \ and}\quad S=\{0\}\cup\bigl(\tfrac{1}{2},1\bigr].
\]
\emph{Define}
\[
f(s,x)=%
\begin{cases}
-(x-s)^{2}, & s\in\bigl(\tfrac{1}{2},1\bigr]\\
\tfrac{1}{2}-x, & s=0.
\end{cases}
\]
\emph{Notice that (i) is satisfied since the sections} $f_{s}$ \emph{are
smooth for every} $s\in S$. \emph{It easy to check that}
\[
\sigma\left(  x\right)  =\left\{
\begin{array}
[c]{cc}%
\left\{  0\right\}  , & x\in\left[  0,1/2\right]  \\
\left\{  x\right\}  , & x\in\left(  1/2,1\right]
\end{array}
\right.
\]
\emph{Hence, the value function will be}%
\[
v\left(  x\right)  =\left\{
\begin{array}
[c]{cc}%
\tfrac{1}{2}-x & x\in\left[  0,1/2\right]  \\
0 & x\in\left(  1/2,1\right]
\end{array}
\right.
\]
\emph{Observe that (ii) is violated. Indeed, setting} $x_{n}=\tfrac{1}%
{2}+\tfrac{t_{n}}{2}\quad(t_{n}\downarrow0)$\emph{, we have}
\begin{align*}
\limsup_{n\rightarrow\infty}\frac{f\bigl(s_{n},x_{n}\bigr)-f\bigl(s_{n}%
,\tfrac{1}{2}\bigr)}{t_{n}} &  =\limsup_{n\rightarrow\infty}\frac
{0-\bigl[-\bigl(x_{n}-\tfrac{1}{2}\bigr)^{2}\bigr]}{t_{n}}\\[6pt]
&  \left.  =\limsup_{n\rightarrow\infty}\frac{t_{n}^{2}/4}{t_{n}}=0>-\tfrac
{1}{2}=\sup_{s\in\sigma(\tfrac{1}{2})}Df_{s}\bigl(\tfrac{1}{2};1\bigr)\right. .
\end{align*}
\emph{Note further that}
\[
Dv\Bigl(\tfrac{1}{2};1\Bigr)=0,\qquad\sup_{s\in\sigma(\tfrac{1}{2})}%
Df_{s}\Bigl(\tfrac{1}{2};1\Bigr)=-\tfrac{1}{2},
\]
\emph{hence }%
\[
Dv\bigl(\tfrac{1}{2};1\bigr)\neq\sup_{s\in\sigma(\tfrac{1}{2})}Df_{s}%
\bigl(\tfrac{1}{2};1\bigr)
\]
\hfill$\blacktriangle$
\end{example}
In Section \ref{bayesianrobustness}, we provide an application of this theorem
in statistics in which both assumptions (i) and (ii) are satisfied (see
Proposition \ref{prop:devrob}).

However condition (ii) may be sometimes hard to check. To clarify this point
we offer two sufficient conditions for the validity of (ii).

The first one assumes that the section functions $f_{s}$ are a-differentiable
with the gradients $\nabla_{a}f_{s}\left(  x\right)  $ being linearly
continuous, uniformly respect to the elements of the space $S.$ More
precisely, for all $x,y\in C$ the limit
\begin{equation}
\left\langle y-x,\nabla_{a}f_{s}\left(  \left(  1-\tau\right)  x+\tau
y\right)  \right\rangle \rightarrow\left\langle y-x,\nabla_{a}f_{s}\left(
x\right)  \right\rangle \label{unif}%
\end{equation}
as $\tau\rightarrow0^{+}$, must hold \emph{uniformly} over $s\in S$. Notice
that unlike $\nabla_{a}f_{s},$ the quantities
\[
\left\langle y-x,\nabla_{a}f_{s}\left(  \left(  1-\tau\right)  x+\tau
y\right)  \right\rangle ,
\]
are not affected by the elements chosen in the class $\left[  \nabla_{a}%
f_{s}\right]  $.

\begin{proposition}
Let $f_{s}$ be a-differentiable. Under (\ref{unif}) condition (ii) of Theorem
\ref{thm:danskin} holds.
\end{proposition}

\noindent\textbf{Proof} Using (\ref{unif}), we can apply the mean value theorem
(Theorem \ref{cor:mean}) to the difference $f(s_{n},x_{t_{n}})-f(s_{n},x)\equiv$ $f_{s_{n}%
}\left(  x_{t_{n}}\right)  -f_{s_{n}}\left(  x\right)  $. It holds%
\[
f_{s_{n}}\left(  x_{t_{n}}\right)  -f_{s_{n}}\left(  x\right)  =\left\langle
t_{n}(y-x),\nabla_{a}f_{s_{n}}\left(  x_{\tau_{n}t_{n}}\right)  \right\rangle
\]
for a certain scalar sequence $\tau_{n}\in\left[  0,1\right]  .$
Consequently,
\[
\frac{f(s_{n},x_{t_{n}})-f(s_{n},x)}{t_{n}}=\left\langle y-x,\nabla
_{a}f_{s_{n}}\left(  x_{\tau_{n}t_{n}}\right)  \right\rangle
\]
We can write
\begin{align*}
\left\langle y-x,\nabla_{a}f_{s_{n}}\left(  x_{\tau_{n}t_{n}}\right)
\right\rangle  &  =\left\langle y-x,\nabla_{a}f_{s_{n}}\left(  x\right)
\right\rangle +\left\langle y-x,\nabla_{a}f_{s_{n}}\left(  x_{\tau_{n}t_{n}%
}\right)  -\nabla_{a}f_{s_{n}}\left(  x\right)  \right\rangle \\
&  =Df_{s_{n}}\left(  x;y\right)  +\left\langle y-x,\nabla_{a}f_{s_{n}}\left(
x_{\tau_{n}t_{n}}\right)  -\nabla_{a}f_{s_{n}}\left(  x\right)  \right\rangle
\\
&  \leq\sup_{s\in\sigma\left(  x\right)  }Df_{s}(x;y)+\left\langle
y-x,\nabla_{a}f_{s_{n}}\left(  x_{\tau_{n}t_{n}}\right)  -\nabla_{a}f_{s_{n}%
}\left(  x\right)  \right\rangle
\end{align*}
Therefore
\[
\limsup_{n\rightarrow\infty}\frac{f(s_{n},x_{t_{n}})-f(s_{n},x)}{t_{n}}%
\leq\sup_{s\in\sigma\left(  x\right)  }Df_{s}(x;y)+\limsup_{n\rightarrow
\infty}\left\langle y-x,\nabla_{a}f_{s_{n}}\left(  x_{\tau_{n}t_{n}}\right)
-\nabla_{a}f_{s_{n}}\left(  x\right)  \right\rangle
\]
If the last term vanishes as $n\rightarrow\infty$, we are done since the
inequality (\ref{eq:FFF}) obtains. On the other hand, using (\ref{unif}),
for an arbitrary $\varepsilon>0$ we have
\[
\left\vert \left\langle y-x,\nabla_{a}f_{s}\left(  x_{\tau_{n}t_{n}}\right)
-\nabla_{a}f_{s}\left(  x\right)  \right\rangle \right\vert \leq\varepsilon
\]
for $n\geq n_{0}$, for some $n_{0}$ and for all $s\in S$. In particular,
$\left\vert \left\langle y-x,\nabla_{a}f_{s_{n}}\left(  x_{\tau_{n}t_{n}%
}\right)  -\nabla_{a}f_{s_{n}}\left(  x\right)  \right\rangle \right\vert
\leq\varepsilon$ for $n\geq n_{0}$ and so $\lim_{n\rightarrow\infty
}\left\langle y-x,\nabla_{a}f_{s_{n}}\left(  x_{\tau_{n}t_{n}}\right)
-\nabla_{a}f_{s_{n}}\left(  x\right)  \right\rangle =0.$\hfill$\blacksquare$

\bigskip

The second one assumes the space $S$ to be a finite set of indices $I$ and
thus $v$ is the maximum of finitely many functions, that is, $v\left(
x\right)  =\max_{i\in I}f_{i}\left(  x\right)  $. \ Here condition (ii) is
always satisfied.

\begin{proposition}
If the affine directional derivatives $Df_{i}(x;y)$ exist for all $i\in
\sigma\left(  x\right)  \subseteq I$, then the affine directional derivative
$Dv(x;y)$ exists and $Dv(x;y)=\max_{i\in\sigma\left(  x\right)  }Df_{i}(x;y)$.
\end{proposition}

\noindent\textbf{Proof} Let $\left\{  t_{n}\right\}  $ and $i_{n}\in
\sigma\left(  x_{t_{n}}\right)  \in I$ be two sequences as claimed in the
condition (ii) of Theorem \ref{thm:danskin}. There exist subsequences
$\left\{  t_{n_{k}}\right\}  $ and $\left\{  i_{n_{k}}\right\}  $ for which
\[
\limsup_{n\rightarrow\infty}\frac{f(i_{n},x_{t_{n}})-f(i_{n},x)}{t_{n}}%
=\lim_{k\rightarrow\infty}\frac{f(i_{n_{k}},x_{t_{n_{k}}})-f(i_{n_{k}}%
,x)}{t_{n_{k}}}%
\]

As $I$ is finite, the sequence $i_{n_{k}}$ is eventually constant and equal to
a point $i^{\ast}\in I$. Consequently, the last limit can be replaced by%
\[
\lim_{k\rightarrow\infty}\frac{f(i^{\ast},x_{t_{n_{k}}})-f(i^{\ast}%
,x)}{t_{n_{k}}}=Df_{i^{\ast}}(x)\leq\max_{i\in\sigma\left(  x\right)  }%
Df_{i}(x;y)
\]
\hfill$\blacksquare$

\section{Examples of affinely differentiable functionals in Economics and
Statistics\label{ch3}}

\subsection{Influence curve\label{sect:inf}}

Let $T:\Delta\left(  Y\right)  \rightarrow\mathbb{R}$ be wa-differentiable.
Using the identification $x\mapsto\delta_{x}$, based on Hampel \cite{hamp} we
define the \emph{influence function} $\mathcal{I}\left(  \cdot;T,\mu\right)
:Y\rightarrow\mathbb{R}$ by%
\[
\mathcal{I}\left(  x;T,\mu\right)  =DT\left(  \mu;\delta_{x}\right)
=\lim_{t\downarrow0}\frac{T\left(  \left(  1-t\right)  \mu+t\delta_{x}\right)
-T\left(  \mu\right)  }{t}%
\]
The name originated in developing robust statistics. It measures the change of
the value $T\left(  \mu\right)  $ when an infinitesimally small part of $\mu$
is replaced by a point-mass at $x.$

When $T$ is a-differentiable,%
\[
\mathcal{I}\left(  x;T,\mu\right)  =\int u_{\mu}\mathrm{d}\left(  \delta
_{x}-\mu\right)  =u_{\mu}\left(  x\right)  -\int u_{\mu}d\mu
\]
where $u_{\mu}\in C_{b}\left(  Y\right)  $ is an a-gradient, i.e., $u_{\mu}%
\in\left[  \nabla_{a}T\left(  \mu\right)  \right]  $ (see Example
\ref{ex:lop}). Since the gradient $u_{\mu}$ is unique up to an additive
constant, under the normalization $\int u_{\mu}\mathrm{d}\mu=0$ we get%
\begin{equation}
\mathcal{I}\left(  \cdot;T,\mu\right)  =u_{\mu}\left(  \cdot\right)  \in
C_{b}\left(  Y\right)  \label{eq:lala}%
\end{equation}
The influence function thus completely pins down the a-differential of the
functional $T$ as $\mu$ varies. Conversely, the influence function agrees with
the gradient under the normalization condition $\int u_{\mu}\mathrm{d}\mu=0$.

\begin{example}
\label{diniexample}\emph{Let} $\mathcal{D}\left[  0,1\right]  $ \emph{be the
space of the probability distributions on the interval} $\left[  0,1\right]
$. \emph{The affine functional} $T:\mathcal{D}\left[  0,1\right]
\rightarrow\mathbb{R}$ \emph{defined by }$T\left(  F\right)  =F\left(
x_{0}\right)  $,\emph{ with }$x_{0}\in\left(  0,1\right)  $,\emph{ is clearly
wa-differentiable and }$DT\left(  F;G\right)  =G\left(  x_{0}\right)
-F\left(  x_{0}\right)  $\emph{. The influence function is given by}%
\[
\mathcal{I}\left(  x;T,F\right)  =G_{x}\left(  x_{0}\right)  -F\left(
x_{0}\right)
\]
\emph{where }$G_{x}$\emph{ is the distribution function of the Dirac measure
}$\delta_{x}$\emph{. Clearly, it is discontinuous at} $x_{0}$\emph{. By
(\ref{eq:lala}), it should coincide with the continuous function }$u_{F}%
$\emph{. Hence, }$T$ \emph{is not a-differentiable at any distribution}
$F\in\mathcal{D}\left[  0,1\right]  $\emph{.}\hfill$\blacktriangle$
\end{example}

Thus, the existence of the influence function does not ensure
a-differentiability. Next we show that this is the case even when the
influence function is continuous (see\ also Example 2.2.3 of \cite{fern}).

\begin{example}
\emph{Define the convex functional }$T:\mathcal{D}\left[  0,1\right]
\rightarrow\mathbb{R}$ \emph{by}%
\begin{equation}
T\left(  F\right)  =\sum_{x\in\left[  0,1\right]  }\left[  F\left(  x\right)
-F\left(  x-\right)  \right]  ^{\alpha} \label{eq:did}%
\end{equation}
\emph{with }$\alpha>1$\emph{. This functional measures the jumps of the
distribution }$F$\emph{. This series is well defined since there are at most
countably many jumps. Therefore, ,}
\[
\sum_{x\in\left[  0,1\right]  }\left[  F\left(  x\right)  -F\left(  x-\right)
\right]  ^{\alpha}\leq\sum_{x\in\left[  0,1\right]  }\left[  F\left(
x\right)  -F\left(  x-\right)  \right]  \leq1
\]
\emph{and} \emph{the sum is finite.} \emph{It is easy to see that }$T$\emph{
is wa-differentiable and its wa-derivative is}%
\[
DT\left(  F;G_{x}\right)  =-\alpha T\left(  F\right)  +\alpha\left(  F\left(
x\right)  -F\left(  x-\right)  \right)  ^{\alpha-1}%
\]
\emph{If the distribution }$F$ \emph{is continuous, then }$DT\left(
F;G_{x}\right)  =0$. \emph{Hence the gradient should vanish everywhere and, by
Theorem \ref{cor:mean}}, $F$\emph{ would be constant, a contradiction.
Consequently, }$T$ \emph{is not a-differentiable at }$F$\emph{.}
\hfill$\blacktriangle$
\end{example}

\subsection{Multi-utility representations}

Let us interpret an element $\mu\in\Delta\left(  Y\right)  $ as a lottery with
prizes in a metric space $Y$. A decision maker (DM) has a preference (binary)
relation $\succsim$ over $\Delta\left(  Y\right)  $ represented by a utility
function $U:\Delta\left(  Y\right)  \rightarrow\mathbb{R}$. Namely,%
\[
\mu\succsim\lambda\Longleftrightarrow U\left(  \mu\right)  \geq U\left(
\lambda\right)
\]
We introduce the auxiliary subrelation $\succsim^{\ast}$ defined by%
\[
\mu\succsim^{\ast}\lambda\Longleftrightarrow\alpha\mu+\left(  1-\alpha\right)
\nu\succsim\alpha\lambda+\left(  1-\alpha\right)  \nu
\]
$\text{for all }\alpha\in\left[  0,1\right]  \text{ and all }\nu\in
\Delta\left(  Y\right)  $. It captures the comparison over which the DM feels
sure (see \cite{gmm} and \cite{cerdilor}).

In the next theorem $\nabla_{a}U\left(  \mu\right)  $ is understood to be a
normalized gradient. For instance, $\nabla_{a}U\left(  \mu\right)  =u_{\mu}$
with $\int u_{\mu}\mathrm{d}\mu=0$, that is, the influence function associated
with the probability measure $\mu$. Moreover, denote by $\operatorname{Im}%
\nabla_{a}U\subseteq C_{b}\left(  Y\right)  $ the image $\left\{  \nabla
_{a}U\left(  \lambda\right)  :\lambda\in\Delta\left(  Y\right)  \right\}  $.

\begin{theorem}
\label{th:jop}If $\succsim$ is a preference relation with an a-differentiable
utility function $U$, then
\[
\mu\succsim^{\ast}\lambda\Longleftrightarrow\int u\mathrm{d}\mu\geq\int
u\mathrm{d}\lambda\qquad\forall u\in\operatorname{Im}\nabla_{a}U
\]

\end{theorem}

In risk theory, since Machina \cite{mach}, the functions $u\in
\operatorname{Im}\nabla_{a}U$ are known as \emph{local utilities}.\emph{
}Theorem \ref{th:jop} formalizes the idea that, individually, each local
utility models a local expected utility behavior of $\succsim$, but jointly
all local utilities characterize a global expected utility feature of
$\succsim$.

\bigskip

\noindent\textbf{Proof} Define $\hat{\succsim}$ by%
\[
\mu\,\hat{\succsim}\,\lambda\Longleftrightarrow\int u\mathrm{d}\mu\geq\int
u\mathrm{d}\lambda\qquad\forall u\in\operatorname{Im}\nabla_{a}U
\]
Assume that $\mu\,\hat{\succsim}\,\lambda$. Let $\nu\in\Delta\left(  Y\right)
$ and $\alpha\in\left(  0,1\right)  $. By the mean value theorem (Theorem
\ref{cor:mean}),
\[
d=U\left(  \alpha\mu+\left(  1-\alpha\right)  \nu\right)  -U\left(
\alpha\lambda+\left(  1-\alpha\right)  \nu\right)  =\alpha\left\langle
\nabla_{a}U\left(  \zeta\right)  ,\mu-\lambda\right\rangle
\]
where $\zeta=\alpha\left[  \left(  1-t\right)  \mu+t\lambda\right]  +\left(
1-\alpha\right)  \nu$ for some $t\in\left(  0,1\right)  $. Since $\mu
\,\hat{\succsim}\,\lambda$, by setting $u_{\zeta}=\nabla_{a}U\left(
\zeta\right)  $ we have%
\[
d=\alpha\left[  \int u_{\zeta}\mathrm{d}\mu-\int u_{\zeta}\mathrm{d}%
\lambda\right]  \geq0
\]
that is, $\mu\succsim^{\ast}\lambda$. Conversely, assume $\mu\succsim^{\ast
}\lambda$ and pick any element $\nu\in\Delta\left(  Y\right)  $. It follows
that, for each $t\in\left[  0,1\right]  $,%
\[
\left(  1-t\right)  \nu+t\mu\succsim\left(  1-t\right)  \nu+t\lambda
\]
Hence, for $t>0$,
\[
\frac{U\left(  \left(  1-t\right)  \nu+t\mu\right)  -U\left(  \nu\right)  }%
{t}\geq\frac{U\left(  \left(  1-t\right)  \nu+t\lambda\right)  -U\left(
\nu\right)  }{t}%
\]
Letting $t\downarrow0$, we get $\left\langle \nabla_{a}U\left(  \nu\right)
,\mu-\nu\right\rangle \geq\left\langle \nabla_{a}U\left(  \nu\right)
,\lambda-\nu\right\rangle $. Namely, $\left\langle \nabla_{a}U\left(
\nu\right)  ,\mu\right\rangle \geq\left\langle \nabla_{a}U\left(  \nu\right)
,\lambda\right\rangle $. That is, $\mu\,\hat{\succsim}\,\lambda$ as
desired.\hfill$\blacksquare$

\bigskip

The next example is based on \cite{chew}.

\begin{example}
\label{ex:epstein}\emph{Define the utility function }$U:\Delta\left(
Y\right)  \rightarrow\mathbb{R}$ \emph{by the quadratic functional}
\[
U\left(  \mu\right)  =\int\psi\left(  x,y\right)  \mathrm{d}\mu\left(
x\right)  \otimes\mathrm{d}\mu\left(  y\right)
\]
\emph{where }$\psi\left(  x,y\right)  $\emph{ is a symmetric, continuous and
bounded function on }$Y\times Y$\emph{. In view of (\ref{eq:quad}) }%
\[
DU\left(  \mu;\lambda\right)  =2\int\psi\left(  x,y\right)  \mathrm{d}%
\mu\left(  x\right)  \otimes\mathrm{d}\lambda\left(  y\right)  -2U\left(
\mu\right)
\]
\emph{and }$U$ \emph{is a-differentiable with gradient }
\[
\nabla_{a}U\left(  \mu\right)  =u_{\mu}=2\int\psi\left(  \cdot,y\right)
\mathrm{d}\mu\left(  y\right)
\]
\emph{up to a constant. Therefore,}%
\[
\operatorname{Im}\nabla_{a}U=\left\{  2\int\psi\left(  \cdot,y\right)
\mathrm{d}\nu\left(  y\right)  :\nu\in\Delta\left(  X\right)  \right\}
\]
\emph{By} \emph{Theorem \ref{th:jop},}%
\[
\mu\succsim^{\ast}\lambda\Longleftrightarrow\int\psi\left(  x,y\right)
\mathrm{d}\mu\left(  x\right)  \otimes\mathrm{d}\nu\left(  y\right)  \geq
\int\psi\left(  x,y\right)  \mathrm{d}\lambda\left(  x\right)  \otimes
\mathrm{d}\nu\left(  y\right)
\]
\emph{for all} $\nu\in\Delta\left(  Y\right)  $\emph{. Clearly,} \emph{this is
equivalent to}
\[
\mu\succsim^{\ast}\lambda\Longleftrightarrow\int\psi\left(  x,y\right)
\mathrm{d}\mu\left(  x\right)  \geq\int\psi\left(  x,y\right)  \mathrm{d}%
\lambda\left(  x\right)  \qquad\forall y\in X
\]
\hfill$\blacktriangle$
\end{example}

\subsection{Prospect theory}

To each monetary lottery $\mu\in\Delta\left(  \mathbb{R}\right)  $ we
associate its distribution function $F_{\mu}\left(  x\right)  =\mu\left(
\left(  -\infty,x\right]  \right)  $. Fix a nonatomic positive Borel measure
$\rho$ on the real line and consider the utility function $U:\Delta\left(
\mathbb{R}\right)  \rightarrow\mathbb{R}$ given by%
\begin{equation}
U\left(  \mu\right)  =\int_{\left[  0,\infty\right]  }w_{+}\left(  1-F_{\mu
}\left(  x\right)  \right)  \mathrm{d}\rho\left(  x\right)  -\int_{\left[
-\infty,0\right]  }w_{-}\left(  F_{\mu}\left(  x\right)  \right)
\mathrm{d}\rho\left(  x\right)  \label{eq:baba}%
\end{equation}
where $w_{+},w_{-}:\left[  0,1\right]  \rightarrow\left[  0,1\right]  $ are
two strictly increasing and continuously differentiable maps. Instances of
this utility function appear in the Prospect Theory of Tverski and Kahneman
\cite{kantv} (see Wakker \cite{wakker}). The utility function $U$ is
wa-differentiable:\footnote{Continuous differentiability of $w_{+}$ and
$w_{-}$ implies their Lipschitzianity, so we can apply the Dominated
Convergence Theorem.}%
\begin{align*}
&  DU\left(  \mu;\lambda\right)  =\lim_{t\downarrow0}\frac{U\left(
\mu+t\left(  \lambda-\mu\right)  \right)  -U\left(  \mu\right)  }{t}\\
&  =-\int_{\left[  0,+\infty\right]  }w_{+}^{\prime}\left(  1-F_{\mu}\left(
x\right)  \right)  \left(  F_{\lambda}-F_{\mu}\right)  \mathrm{d}\rho\left(
x\right)  -\int_{\left[  -\infty,0\right]  }w_{-}^{\prime}\left(  F_{\mu
}\left(  x\right)  \right)  \left(  F_{\lambda}-F_{\mu}\right)  \mathrm{d}%
\rho\left(  x\right) \\
&  =\int_{\mathbb{R}}\varphi_{\mu}\left(  x\right)  \left(  F_{\mu}\left(
x\right)  -F_{\lambda}\left(  x\right)  \right)  \mathrm{d}\rho\left(
x\right)
\end{align*}
where $\varphi$ is the bounded scalar function%
\[
\varphi_{\mu}\left(  x\right)  =\left\{
\begin{tabular}
[c]{ll}%
$\medskip w_{+}^{\prime}\left(  1-F_{\mu}\left(  x\right)  \right)  $ & if
$x\geq0$\\
$w_{-}^{\prime}\left(  F_{\mu}\left(  x\right)  \right)  $ & $\text{otherwise}%
$%
\end{tabular}
\right.
\]
It is not immediately obvious whether $DU\left(  \mu;\cdot\right)  $ is extendable or not.
To prove this we must use integration by parts for Riemann-Stieltjes
integrals. Set $\mathrm{d}\eta=\varphi_{\mu}\mathrm{d}\rho$ and integrate
first on a finite interval. By the Dominated Convergence Theorem,%
\[
\int_{\mathbb{R}}\varphi_{\mu}\left(  x\right)  \left(  F_{\mu}\left(
x\right)  -F_{\lambda}\left(  x\right)  \right)  \mathrm{d}\rho\left(
x\right)  =\lim_{R\rightarrow+\infty}\int_{\left[  -R,R\right]  }\varphi_{\mu
}\left(  x\right)  \left(  F_{\mu}\left(  x\right)  -F_{\lambda}\left(
x\right)  \right)  \mathrm{d}\rho\left(  x\right)
\]
Hence,%
\[
\int_{\left[  -R,R\right]  }\varphi_{\mu}\left(  x\right)  \left(  F_{\mu
}\left(  x\right)  -F_{\lambda}\left(  x\right)  \right)  \mathrm{d}%
\rho\left(  x\right)  =\int_{\left[  -R,R\right]  }\left(  F_{\mu}\left(
x\right)  -F_{\lambda}\left(  x\right)  \right)  \mathrm{d}\eta\left(
x\right)  =\int_{-R}^{R}\left(  F_{\mu}\left(  x\right)  -F_{\lambda}\left(
x\right)  \right)  \mathrm{d}\Phi_{\mu}\left(  x\right)
\]
where the last integral is Riemann-Stieltjes, and the function $\Phi_{\mu
}:\mathbb{R}\rightarrow\mathbb{R}$ defined by
\begin{equation}
\Phi_{\mu}\left(  x\right)  =\int_{\left[  -\infty,x\right]  }\varphi_{\mu
}\left(  t\right)  \mathrm{d}\rho\left(  t\right)  \label{eq:form}%
\end{equation}
is bounded and absolutely continuous. By the integration-by-parts formula,%
\begin{align*}
&  \int_{-R}^{R}\left(  F_{\mu}\left(  x\right)  -F_{\lambda}\left(  x\right)
\right)  \mathrm{d}\Phi_{\mu}\left(  x\right)  =\left[  \left(  F_{\mu}\left(
x\right)  -F_{\lambda}\left(  x\right)  \right)  \Phi_{\mu}\left(  x\right)
\right]  _{-R}^{R}-\int_{-R}^{R}\Phi_{\mu}\left(  x\right)  \mathrm{d}\left(
F_{\mu}-F_{\lambda}\right) \\
&  =\left[  \left(  F_{\lambda}\left(  x\right)  -F_{\mu}\left(  x\right)
\right)  \Phi_{\mu}\left(  x\right)  \right]  _{-R}^{R}+\int_{\left[
-R,R\right]  }\Phi_{\mu}\left(  x\right)  \mathrm{d}\left(  \lambda
-\mu\right)
\end{align*}
Letting $R\rightarrow\infty$, we finally get%
\[
DV\left(  \mu;\lambda\right)  =\int_{\mathbb{R}}\Phi_{\mu}\left(  x\right)
\mathrm{d}\left(  \lambda-\mu\right)  =\left\langle \Phi_{\mu},\lambda
-\mu\right\rangle
\]
since $\left[  \left(  F_{\lambda}\left(  x\right)  -F_{\mu}\left(  x\right)
\right)  \Phi_{\mu}\left(  x\right)  \right]  _{-R}^{R}\rightarrow0$ as
$\Phi_{\mu}$ is bounded.

In sum, $U$ is a-differentiable and its gradient is
\[
\nabla_{a}U\left(  \mu\right)  =\Phi_{\mu}\in C_{b}\left(  \mathbb{R}\right)
\text{ }%
\]
Let $\succsim$ be a preference relation on $\Delta\left(  \mathbb{R}\right)  $
represented by $U$. For its subrelation $\succsim^{\ast}$ we thus have the
following consequence of Theorem \ref{th:jop},

\begin{proposition}
Let $\mu,\lambda\in\Delta\left(  \mathbb{R}\right)  $. It holds, for each
$\nu\in\Delta\left(  \mathbb{R}\right)  $,
\begin{equation}
\mu\succsim^{\ast}\lambda\Longleftrightarrow\int_{\mathbb{R}}\varphi_{\nu
}\left(  x\right)  \left(  1-F_{\mu}\left(  x\right)  \right)  \mathrm{d}%
\rho\left(  x\right)  \geq\int_{\mathbb{R}}\varphi_{\nu}\left(  x\right)
\left(  1-F_{\lambda}\left(  x\right)  \right)  \mathrm{d}\rho\left(
x\right)  \label{eq:gaga}%
\end{equation}

\end{proposition}

\noindent\textbf{Proof} By Theorem \ref{th:jop}, $\mu\succsim^{\ast}\lambda$
if and only if, for each $\nu\in\Delta\left(  \mathbb{R}\right)  $, $\int
\Phi_{v}\left(  t\right)  \mathrm{d}\mu\left(  t\right)  \geq\int\Phi
_{v}\left(  t\right)  \mathrm{d}\lambda\left(  t\right)  $. By Fubini's
Theorem,
\[
\int\Phi_{v}\left(  t\right)  \mathrm{d}\mu\left(  t\right)  =\int
\mathrm{d}\mu\left(  t\right)  \int I_{\left[  -\infty,t\right]  }\left(
x\right)  \varphi_{\nu}\left(  x\right)  \mathrm{d}\rho\left(  x\right)
=\int\mathrm{d}\rho\left(  x\right)  \varphi_{\nu}\left(  x\right)  \int
I_{\left[  x,+\infty\right]  }\left(  t\right)  \mathrm{d}\mu\left(  t\right)
\]
where we used the relation $I_{\left[  -\infty,t\right]  }\left(  x\right)
=I_{\left[  x,+\infty\right]  }\left(  t\right)  $. Hence,%
\[
\int\Phi_{v}\left(  t\right)  \mathrm{d}\mu\left(  t\right)  =\int\varphi
_{\nu}\left(  x\right)  \left(  1-F_{\mu}\left(  x-\right)  \right)
\mathrm{d}\rho\left(  x\right)  =\int\varphi_{\nu}\left(  x\right)  \left(
1-F_{\mu}\left(  x\right)  \right)  \mathrm{d}\rho\left(  x\right)
\]
where the last equality is true because $\rho$ is nonatomic. This proves
(\ref{eq:gaga}).\hfill$\blacksquare$

\subsection{Bayesian robustness\label{bayesianrobustness}}

Let $\Theta$ be a parameter space and $\mathcal{X}$ a sample space. For any
given sample $x\in\mathcal{X}$, a \emph{posterior functional} $\rho_{x}%
:\Delta\left(  \Theta\right)  \rightarrow\mathbb{R}$ maps a prior distribution
$\mu\in\Delta(\Theta)$ to a scalar representing a posterior statistic of
interest. For example, the posterior mean is given by $\rho_{x}(\mu
)=\int\theta\mathrm{d}\mu_{x}(\theta)$, with $\mu_{x}$ being the Bayesian
update of $\mu$.

Bayesian robustness (see, e.g., \cite{berg1} and \cite{berg2}) investigate how
posterior outcomes vary under different prior specifications, often by
examining the range of $\rho_{x}$ over a set of priors $\mathcal{M}%
\subseteq\Delta\left(  \Theta\right)  $. When $\mathcal{M}$ is convex, our
methods can be used to compute such a range. Indeed, if $\mu_{1},\mu_{2}%
\in\Delta\left(  \Theta\right)  $ are such that $\rho_{x}(\mathcal{M}%
)=[\rho_{x}(\mu_{1}),\rho_{x}(\mu_{2})]$, according to (\ref{eq:SSS}) it
follows that
\[
\min_{\nu\in\mathcal{M}}D\rho_{x}(\mu_{1};\nu)=\max_{\nu\in\mathcal{M}}%
D\rho_{x}(\mu_{2};\nu)=0
\]
These two necessary conditions can be used to develop numerical algorithms (we
refer to \cite{basu}).

Another statistical functional of interest is the expected posterior loss of a
Bayesian estimator. Formally, denote by $s\in S$ the choice of the estimator.
Given a loss function $\ell:S\times\Theta\rightarrow\mathbb{R}$ and a prior
$\mu\in\Delta(\Theta)$, the expected loss $L(\mu)$ of the Bayesian estimator
is
\[
L(\mu)=\inf_{s\in S}\int\ell(s,\theta)\mathrm{d}\mu(\theta)
\]
Given an alternative prior $\nu\in\Delta(\Theta)$, the directional derivative
$DL(\mu;\nu)$ captures the sensitivity of the estimator to the prior $\mu$
(see \cite{kchuang} and also \cite{stanca} for applications in economics). In
the rest of this subsection we show that, thanks to Theorem \ref{thm:danskin},
we are able to compute $DL(\mu;\nu)$ in important cases of interest, allowing
for an unbounded set of parameters (as common in applications).

So, let\emph{ }$S=\Theta=\mathbb{R}$. Consider\emph{ }$\mu,\nu\in
\Delta(\mathbb{R})$\emph{ }with finite first moment. In order to apply Theorem
\ref{thm:danskin} we study the related optimization problem%
\[
L^{\ast}\left(  \mu\right)  =\sup_{s\in\mathbb{R}}U\left(  s,\mu\right)
\]
where\emph{ }$U\left(  s,\mu\right)  =-\int\left\vert s-\theta\right\vert
d\mu\left(  \theta\right)  $.

\begin{proposition}
\label{prop:devrob} The affine directional derivative $DL^{\ast}\left(
\mu;\nu\right)  $ exists. More specifically,%
\begin{equation}
DL^{\ast}\left(  \mu;\nu\right)  =\sup_{s\in\sigma\left(  \mu\right)  }%
DU_{s}\left(  \mu;\nu\right)  =\sup_{s\in\sigma\left(  \mu\right)  }%
\int|s-\theta|\mathrm{d}\mu(\theta)-\int|s-\theta|\mathrm{d}\nu(\theta)
\label{eq:danskin}%
\end{equation}

\end{proposition}

\noindent\textbf{Proof} It is well known that given $\lambda\in\Delta
(\mathbb{R)}$ the set $\sigma\left(  \lambda\right)  $ of maximizers of the
section $U_{s}\equiv U\left(  s,\cdot\right)  $ is a median value of the
distribution $F_{\lambda}$. Specifically, a point $s\in\sigma\left(
\lambda\right)  $ is characterized by the equations%
\begin{equation}
F_{\lambda}\left(  s\right)  \geq\frac{1}{2}\quad\text{and}\quad F_{\lambda
}\left(  s^{-}\right)  \leq\frac{1}{2} \label{eq:JJJ}%
\end{equation}
The set $%
{\textstyle\bigcup_{t\in\left[  0,1\right]  }}
\sigma\left(  \left(  1-t\right)  \mu+t\nu\right)  $ is contained into a
compact interval $K$ of $\mathbb{R}$ (just pick any compact interval
containing the sets of points $1/2-\eta\leq F_{\mu}(x)\leq1/2+\eta$ and
$1/2-\eta\leq F_{\nu}(x)\leq1/2+\eta$ for some $0<\eta<1/2$). By convexity,
$\sigma\left(  \left(  1-t\right)  \mu+t\nu\right)  \subseteq K$ for all
$t\in\left[  0,1\right]  $. Let us show that the two assumptions of Theorem
\ref{thm:danskin} are satisfied. Clearly, $U_{s}$ is wa-differentiable and its
affine derivative is
\[
DU_{s}\left(  \mu;\nu\right)  =\int|s-\theta|\mathrm{d}\mu(\theta
)-\int|s-\theta|\mathrm{d}\nu(\theta)
\]
Hence, (i) of Theorem \ref{thm:danskin} holds. Take $t_{n}\downarrow0$ and let
$\left\{  s_{n}\right\}  \subseteq K$ be any sequence with $s_{n}\in
\sigma\left(  \left(  1-t_{n}\right)  \mu+t_{n}\nu\right)  $. In view of
Theorem \ref{thm:danskin}-(ii), take any subsequence $\left\{  s_{n_{k}%
}\right\}  $ such that the sequence%
\begin{equation}
\frac{U\left(  s_{n_{k}},\left(  1-t_{n_{k}}\right)  \mu+t_{n_{k}}\nu\right)
-U\left(  s_{n_{k}},\mu\right)  }{t_{n_{k}}} \label{eq:HHDD}%
\end{equation}
converges. As $K$ is compact, without loss of generality (passing if needed to
a further subsequence) we can assume that the sequence $\left\{  s_{n_{k}%
}\right\}  $ converges to a point $\overline{s}$.

\bigskip

\noindent\textbf{Claim} It holds $\overline{s}\in\sigma\left(  \mu\right)  $.

\bigskip

\noindent\textbf{Proof} For any $\varepsilon>0$ we have $s_{n}\leq\overline
{s}+\varepsilon$ for $n$ sufficiently large. By monotonicity,
\[
\frac{1}{2}\leq\left(  1-t_{n}\right)  F_{\mu}\left(  s_{n}\right)
+t_{n}F_{\nu}\left(  s_{n}\right)  \leq\left(  1-t_{n}\right)  F_{\mu}\left(
\overline{s}+\varepsilon\right)  +t_{n}F_{\nu}\left(  \overline{s}%
+\varepsilon\right)
\]
As $n\rightarrow\infty$ we get $1/2\leq F_{\mu}\left(  \overline
{s}+\varepsilon\right)  $. Since $F_{\mu}$ is right-continuos we have $1/2\leq
F_{\mu}\left(  \overline{s}\right)  $ which is the first condition of
(\ref{eq:JJJ}). As to the second one, begin now with $s_{n}\geq\overline
{s}-\varepsilon$. Hence, $s_{n}-\eta\geq\overline{s}-\varepsilon-\eta$ for any
$\eta>0$. Thus
\[
\left(  1-t_{n}\right)  F_{\mu}\left(  \overline{s}-\varepsilon-\eta\right)
+t_{n}F_{\nu}\left(  \overline{s}-\varepsilon-\eta\right)  \leq\left(
1-t_{n}\right)  F_{\mu}\left(  s_{n}-\eta\right)  +t_{n}F_{\nu}\left(
s_{n}-\eta\right)  \leq\frac{1}{2}%
\]
As $n\rightarrow\infty$, we get $F_{\mu}\left(  \overline{s}-\varepsilon
-\eta\right)  \leq1/2$ for all $\varepsilon+\eta>0$. Hence, condition
(\ref{eq:JJJ}) is true and so $\overline{s}\in\sigma\left(  \mu\right)
$.\hfill$\square$

\bigskip

By this Claim, we can write%
\begin{align*}
&  \lim_{k\rightarrow\infty}\frac{U\left(  s_{n_{k}},\left(  1-t_{n_{k}%
}\right)  \mu+t_{n_{k}}\nu\right)  -U\left(  s_{n_{k}},\mu\right)  }{t_{n_{k}%
}}=\lim_{k\rightarrow\infty}\int|s_{n_{k}}-\theta|\mathrm{d}\mu(\theta
)-\int|s_{n_{k}}-\theta|\mathrm{d}\nu(\theta)\\
&  =\int|\overline{s}-\theta|\mathrm{d}\mu(\theta)-\int|\overline{s}%
-\theta|\mathrm{d}\nu(\theta)=DU_{\overline{s}}\left(  \mu,\nu\right)
\leq\sup_{s\in\sigma\left(  \mu\right)  }DU_{s}\left(  \mu,\nu\right)
\end{align*}
This is the case for every subsequence $\left\{  s_{n_{k}}\right\}  $ for
which (\ref{eq:HHDD}) has a limit. Consequently,
\[
\lim\sup_{n\rightarrow\infty}\frac{U\left(  s_{n},\left(  1-t_{n}\right)
\mu+t_{n}\nu\right)  -U\left(  s_{n},\mu\right)  }{t_{n}}\leq\sup_{s\in
\sigma\left(  \mu\right)  }DU_{s}\left(  \mu,\nu\right)
\]
and so assumption (ii) holds. By Theorem \ref{thm:danskin} we thus have
(\ref{eq:danskin}).\hfill$\blacksquare$

\section{Extensions\label{ch5}}

\subsection{Hadamard and strict differentiability\label{sect:had}}

Endow the convex set $C$ with a topology finer than (or equal to) the relative
weak topology. This finer topology is assumed to be metrizable by a metric
$\rho$.\footnote{Usually the metric $\rho$ can be extended to the vector space
$X$. In this case the topology $\tau$ is compatible with the vector structure
of the space and $\rho$ $\ $inherits several nice properties. However, this
does not significantly affect our analysis.} This preliminary assumption
allows us to restrict ourself to the following sequential formulation of the
Hadamard directional derivative, adapted to our setting.\footnote{We omit for
sake of simplicity alternative Hadamard formulations related to the so-called
compact differentiability (see \cite{sap}).}

\begin{definition}
The (\emph{affine})\emph{ Hadamard directional} \emph{derivative} of
$f:C\rightarrow\mathbb{\mathbb{R}}$ at $x\in C$ along the direction $y\in C$
is given by%
\[
D_{H}f\left(  x;y\right)  =\lim_{n\rightarrow\infty}\frac{f\left(  \left(
1-t_{n}\right)  x+t_{n}y_{n}\right)  -f\left(  x\right)  }{t_{n}}%
\]
when the limit holds for every sequence $t_{n}\downarrow0$ and every sequence
$\left\{  y_{n}\right\}  $ in $C$ with $\rho\left(  y_{n},y\right)
\rightarrow0$.
\end{definition}

Clearly in this case $D_{H}f\left(  x;y\right)  =Df\left(  x;y\right)  $. The
function $f$ is called \emph{Hadamard wa-differentiable }at\emph{ }$x$ if
$D_{H}f\left(  x;\cdot\right)  $ is affine. It is called \emph{Hadamard
a-differentiable }at $x$ if $D_{H}f\left(  x;\cdot\right)  $ is extendable.
Accordingly, $\nabla_{H}f\left(  x\right)  \in X^{\ast}$ denotes a
representative of the equivalence class of Hadamard gradients.

\begin{definition}
A function $f:C\rightarrow\mathbb{\mathbb{R}}$ is \emph{strictly
wa-differentiable} at $x$ if there exists an affine functional $D_{S}f\left(
x;\cdot\right)  :C\rightarrow\mathbb{R}$ such that, for each $y\in C$,%
\[
D_{S}f\left(  x;y\right)  =\lim_{n\rightarrow\infty}\frac{f\left(  \left(
1-t_{n}\right)  x_{n}+t_{n}y\right)  -f\left(  x\right)  }{t_{n}}%
\]
for every sequence $t_{n}\downarrow0$ and every sequence $\left\{
x_{n}\right\}  \subseteq C$ with $\rho\left(  x_{n},x\right)  \rightarrow0$.
\end{definition}

It is worth mentioning a stronger notion of differentiability that combines
the two previous ones.

\begin{definition}
\label{def:SHAD}A function $f:C\rightarrow\mathbb{\mathbb{R}}$ is
\emph{strictly Hadamard wa-differentiable} at $x$ if there exists an affine
functional $D_{SH}f\left(  x;\cdot\right)  :C\rightarrow\mathbb{R}$ such that,
for each $y\in C$,%
\[
D_{SH}f\left(  x;y\right)  =\lim_{n\rightarrow\infty}\frac{f\left(  \left(
1-t_{n}\right)  x_{n}+t_{n}y_{n}\right)  -f\left(  x_{n}\right)  }{t_{n}}%
\]
for every sequence $t_{n}\downarrow0$ and all sequences $\left\{
x_{n}\right\}  $ and $\left\{  y_{n}\right\}  $ in $C$ with $\rho\left(
x_{n},x\right)  \rightarrow0$ and $\rho\left(  y_{n},x\right)  \rightarrow0$.
\end{definition}

\begin{example}
\emph{The quadratic functional }$Q:\Delta\left(  Y\right)  \rightarrow
\mathbb{R}$ \emph{given by}%
\[
Q\left(  \mu\right)  =\int\psi\left(  x,y\right)  \mathrm{d}\mu\left(
x\right)  \otimes\mathrm{d}\mu\left(  y\right)
\]
\emph{with }$\psi$\emph{ symmetric}, \emph{is strictly Hadamard
wa-differentiable under the Prokhorov metric when }$Y$\emph{ is a Polish
space. Indeed, let }$\mu_{n}\Longrightarrow\mu$\emph{, }$\lambda
_{n}\Longrightarrow\lambda$\emph{ and }$t_{n}\downarrow0$ \emph{(here}
$\Longrightarrow$ \emph{denotes the convergence for the Prokhorov metric)}%
$.$\emph{ Then},%
\[
\frac{Q\left(  \left(  1-t_{n}\right)  \mu_{n}+t_{n}\lambda_{n}\right)
-Q\left(  \mu_{n}\right)  }{t_{n}}=t_{n}Q\left(  \lambda_{n}\right)  +\left(
t_{n}-2\right)  Q\left(  \mu_{n}\right)  +2\left(  1-t_{n}\right)  \int
\psi\left(  x,y\right)  \mathrm{d}\mu_{n}\otimes d\lambda_{n}%
\]
\emph{As }$Y$\emph{ is Polish, }$\mu_{n}\Longrightarrow\mu$ \emph{and}
$\lambda_{n}\Longrightarrow\lambda$ \emph{imply} $\mu_{n}\otimes\lambda
_{n}\Longrightarrow\mu\otimes\lambda$. \emph{Hence, the limit of this ratio
exists, and }$D_{SH}Q\left(  \mu;\lambda\right)  =2\int\psi\left(  x,y\right)
\mathrm{d}\mu\otimes\mathrm{d}\lambda-2Q\left(  \mu\right)  $.\hfill
$\blacktriangle$
\end{example}

This example can be easily generalized by showing that a quadratic functional
$Q:C\rightarrow\mathbb{R}$ is strictly Hadamard wa-differentiable under a
metric $\rho$ when the associated biaffine form $B_{S}$ is $\rho$-continuous
on $C\times C$. To further elaborate, we need to introduce a more stringent
class of metrics.

\begin{definition}
\label{def:con} A metric $\rho$ on $C$ is \emph{convex }if
\[
\rho\left(  x,\alpha y_{1}+(1-\alpha)y_{2}\right)  \leq\alpha\rho\left(
x,y_{1}\right)  +(1-\alpha)\rho\left(  x,y_{2}\right)
\]
for all $x,y_{1},y_{2}\in C$ and all $\alpha\in\left[  0,1\right]  $.
\end{definition}

For instance, the Prokhorov metric on $\Delta\left(  Y\right)  $, with $Y$ a
metric and separable space, is equivalent to the convex Dudley metric (see
Theorem 11.3.3 of \cite{dudley}). With this, next we establish a couple of
useful differentiability criteria.

\begin{proposition}
\label{prop:ztm}Let $\rho$ be convex and $f:C\rightarrow\mathbb{R}$ be
wa-differentiable in a $\rho$-neighborhood of a point $\bar{x}\in C$.

\begin{enumerate}
\item[(i)] If the map $x\mapsto Df\left(  x;y\right)  $ is $\rho$-continuous
at $\bar{x}$ for every $y\in C$, then $f$ is strictly wa-differentiable at
$\bar{x}$.

\item[(ii)] If the map $\left(  x,y\right)  \mapsto Df\left(  x;y\right)  $ is
$\rho$-continuous at $\left(  \bar{x},y\right)  $ for every $y\in C$, then $f$
is strictly Hadamard wa-differentiable at $\bar{x}$.
\end{enumerate}
\end{proposition}

\noindent\textbf{Proof} \ Let $f$ be wa-differentiable on $U_{\varepsilon
}\left(  \bar{x}\right)  =\left\{  x\in C;\text{ }\rho(x,\overline{x}%
)\leq\varepsilon\right\}  $. As $\rho$ is convex, for every $x,y\in C$, we
have
\[
\rho\left(  x_{t},\bar{x}\right)  \leq\left(  1-t\right)  \rho\left(
x,\bar{x}\right)  +t\rho\left(  y,\bar{x}\right)  \leq\rho\left(  x,\bar
{x}\right)  +t\rho\left(  y,\bar{x}\right)  .
\]
Consequently, taking the points $x_{n}$ such that $\rho\left(  x_{n},\bar
{x}\right)  <\varepsilon/2$ and $0<t_{n}<\varepsilon/\left[  2\rho\left(
y,\bar{x}\right)  \right]  $, then the two sequences of points $x_{n}$ and
$x_{n}+t_{n}\left(  y-x_{n}\right)  $ belong to the neighborhood
$U_{\varepsilon}\left(  \bar{x}\right)  $. By mean value Theorem
\ref{cor:mean},%
\[
f\left(  x_{n}+t_{n}\left(  y-x_{n}\right)  \right)  -f\left(  x_{n}\right)
=\frac{1}{1-\tau_{n}}Df\left(  x_{n}+\tau_{n}t_{n}\left(  y-x_{n}\right)
;x_{n}+t_{n}\left(  y-x_{n}\right)  \right)
\]
for some $\tau_{n}\in\left(  0,1\right)  $. On the other hand,%
\[
x_{n}+t_{n}\left(  y-x_{n}\right)  =A_{n}\left[  x_{n}+\tau_{n}t_{n}\left(
y-x_{n}\right)  \right]  +B_{n}y
\]
where%

\[
A_{n}=\left(  \frac{1-t_{n}}{1-\tau_{n}t_{n}}\right)  \text{, }B_{n}=\left(
\frac{\left(  1-\tau_{n}\right)  t_{n}}{1-\tau_{n}t_{n}}\right)  \text{ and
}A_{n}+B_{n}=1.
\]
Therefore, by (\ref{eq:sat}) we get%
\[
\frac{f\left(  x_{n}+t_{n}\left(  y-x_{n}\right)  \right)  -f\left(
x_{n}\right)  }{t_{n}}=\frac{1}{1-\tau_{n}t_{n}}Df\left(  x_{n}+\tau_{n}%
t_{n}\left(  y-x_{n}\right)  ;y\right)
\]
Letting $t_{n}\downarrow0$ and $\rho\left(  x_{n},\bar{x}\right)
\rightarrow0$ the limit exists and $D_{S}f\left(  \bar{x};y\right)  =Df\left(
\bar{x};y\right)  $. The proof of the second statement is similar.\hfill
$\blacksquare$

\bigskip

Here is a kind of converse of the previous statement.

\begin{proposition}
\label{prop:bao}Let $f:C\rightarrow\mathbb{R}$ be strictly Hadamard
wa-differentiable at $\bar{x}\in C$. If the affine directional derivative
$Df\left(  x;y\right)  $ exists in a $\rho$-neighborhood of the point $\bar
{x}$, then $Df\left(  \cdot;\cdot\right)  $ is $\rho$-continuous at $\left(
\bar{x},y\right)  $ for all $y\in C$.
\end{proposition}

Likewise, if $f$ is strictly wa-differentiable at $\bar{x}$, then $Df\left(
\cdot\text{ };y\right)  $ is $\rho$-continuous at $\bar{x}$.

\bigskip

\noindent\textbf{Proof} Fix $y\in C$ and consider two sequences $\left\{
x_{n}\right\}  ,\left\{  y_{n}\right\}  \subseteq C$ so that $\rho\left(
y_{n},y\right)  \rightarrow0$ and $\rho\left(  x_{n},\bar{x}\right)
\rightarrow0$ and such that $\left\{  x_{n}\right\}  $ is contained into the
claimed neighborhood of $\bar{x}$. Fix $\varepsilon>0$. For every $n$ there is
$t_{n}\in\left(  0,1\right)  $ so that
\[
\left\vert \frac{f\left(  \left(  1-t_{n}\right)  x_{n}+t_{n}y_{n}\right)
-f\left(  x_{n}\right)  }{t_{n}}-Df\left(  x_{n};y_{n}\right)  \right\vert
\leq\frac{\varepsilon}{2}.
\]
The sequence $\left\{  t_{n}\right\}  $ can be chosen so that $t_{n}%
\downarrow0$. The hypothesis of strict Hadamard wa-differentiability implies
that, for all $n$ sufficiently large,%
\[
\left\vert \frac{f\left(  \left(  1-t_{n}\right)  x_{n}+t_{n}y_{n}\right)
-f\left(  x_{n}\right)  }{t_{n}}-D_{H}f\left(  \bar{x};y\right)  \right\vert
\leq\frac{\varepsilon}{2}%
\]
that yields $\left\vert Df\left(  x_{n};y_{n}\right)  -D_{H}f\left(  \bar
{x};y\right)  \right\vert \leq\varepsilon$. Hence, $Df\left(  x_{n}%
;y_{n}\right)  \rightarrow D_{H}f\left(  \bar{x};y\right)  $ as $\rho\left(
y_{n},y\right)  \rightarrow0$ and $\rho\left(  x_{n},\bar{x}\right)
\rightarrow0$, which is the desired continuity property.\hfill$\blacksquare$

\bigskip

Next we establish a noteworthy consequence.

\begin{proposition}
If $f:C\rightarrow\mathbb{R}$ is strictly Hadamard wa-differentiable, then it
is $\rho$-continuous.
\end{proposition}

\noindent\textbf{Proof} Let $x\in C$. Take a sequence $\left\{  x_{n}\right\}
\subseteq C$ such that $\rho\left(  x_{n},x\right)  \rightarrow0$. By applying
the mean value theorem (Theorem \ref{cor:mean}) to the two points $x$ and
$x_{n}$, there exists a sequence $\left\{  t_{n}\right\}  \subseteq\left(
0,1\right)  $ such that
\begin{equation}
f\left(  x_{n}\right)  -f\left(  x\right)  =\frac{1}{1-t_{n}}Df\left(  \left(
1-t_{n}\right)  x+t_{n}x_{n};x_{n}\right)  =-\frac{1}{t_{n}}Df\left(  \left(
1-t_{n}\right)  x+t_{n}x_{n};x\right)  \label{eq:VVV}%
\end{equation}
Partition the elements of the sequence $\left\{  t_{n}\right\}  $ so to obtain
two finite or infinite subsequences for which either $t_{n}\in\left(
0,1/2\right)  $ or $t_{n}\in\left[  1/2,1\right)  .$ For the elements that
fall into $\left(  0,1/2\right)  $ we have $\left(  1-t_{n}\right)  ^{-1}<2$.
Analogously, it holds $\left\vert -t_{n}^{-1}\right\vert \leq2$ for the points
of the sequence that fall into $\left[  1/2,1\right)  $. Hence,%
\[
Df\left(  \left(  1-t_{n}\right)  x+t_{n}x_{n};x_{n}\right)  \rightarrow
Df\left(  x;x\right)  =0
\]
as well as
\[
Df\left(  \left(  1-t_{n}\right)  x+t_{n}x_{n};x\right)  \rightarrow Df\left(
x;x\right)  =0
\]
Passing, if needed, to subsequences, in light of (\ref{eq:VVV}) we infer that
$f\left(  x_{n}\right)  -f\left(  x\right)  \rightarrow0$ as $\rho\left(
x_{n},x\right)  \rightarrow0$.\hfill$\blacksquare$

\bigskip

Let us illustrate our Hadamard results with some examples.

\begin{example}
\emph{(i) The functional }$T\left(  F\right)  =F\left(  x_{0}\right)  $\emph{
of Example \ref{diniexample} is wa-differentiable, with }$DT\left(
F;G\right)  =G\left(  x_{0}\right)  -F\left(  x_{0}\right)  $\emph{. By
Proposition \ref{prop:ztm}, under the Prokhorov metric, }$T$\emph{ is strictly
wa-differentiable at every probability distribution }$F$\emph{ that is
continuous at the point }$x_{0}$. \emph{Instead, under the }Kolmogorov
metric\emph{ }$\rho_{\infty}$\emph{,\footnote{That is, $\rho_{\infty}\left(
F,G\right)  =\sup\left\vert F\left(  t\right)  -G\left(  t\right)  \right\vert
$.} the functional }$T$\emph{ is strictly Hadamard wa-differentiable at any
point }$F.$\emph{ Indeed,}
\[
\left\vert DT\left(  F;G\right)  -DT\left(  F_{1};G_{1}\right)  \right\vert
\leq\left\vert G\left(  x_{0}\right)  -G_{1}\left(  x_{0}\right)  \right\vert
+\left\vert F\left(  x_{0}\right)  -F_{1}\left(  x_{0}\right)  \right\vert
\leq\left\Vert G-G_{1}\right\Vert _{\infty}+\left\Vert F-F_{1}\right\Vert
_{\infty}%
\]
\emph{By Proposition \ref{prop:ztm}, the continuity of }$DT\left(  \cdot\text{
};\cdot\right)  $\emph{ delivers the desired result.}

\emph{(ii)} \emph{In light of Example \ref{integral}, it is easy to check
that, if }$F\left(  s,\cdot\right)  $\emph{ is strictly Hadamard
wa-differentiable for the Euclidean metric of }$\mathbb{R}^{n},$ \emph{then
the functional }$I:C\rightarrow\mathbb{R}$\emph{ is strictly Hadamard
wa-differentiable for the uniform metric }$\rho_{\infty}(u,v)=\sup
_{S}\left\Vert u-v\right\Vert _{n}$\emph{.} \hfill$\blacktriangle$
\end{example}

\subsection{Frechet differentiability}

We begin with a notion of differentiability closely related to the uniform
convergence of the directional derivative over bounded sets.

\begin{definition}
A function $f:C\rightarrow\mathbb{R}$ is \emph{boundedly wa-differentiable} at
$x\in C$ if there exists a $\rho$-continuous and affine functional
$A:C\rightarrow\mathbb{R}$ such that%
\begin{equation}
\lim_{n\rightarrow\infty}\left(  \frac{f\left(  \left(  1-t_{n}\right)
x+t_{n}y_{n}\right)  -f\left(  x\right)  }{t_{n}}-A\left(  y_{n}\right)
\right)  =0 \label{eq:dupo}%
\end{equation}
for every sequence $t_{n}\downarrow0$ and every $\rho$-bounded sequence
$\left\{  y_{n}\right\}  $ in $C$.
\end{definition}

By setting $y_{n}=x$ for all $n$, one gets $A\left(  x\right)  =0$. Moreover,
since $\rho$-convergent sequences are $\rho$-bounded, a boundedly
wa-differentiable function at $x\in C$ is also Hadamard wa-differentiable, and
$D_{H}f\left(  x,\cdot\right)  =A\left(  \cdot\right)  $.

\begin{proposition}
\label{prop:gaga}Let $\rho$ be convex. A function $f:C\rightarrow\mathbb{R}$
is boundedly wa-differentiable at $x\in C$ if and only if the limit
(\ref{eq:dupo}) holds for all sequences $\left\{  y_{n}\right\}  $ in some
$\rho$-neighborhood $U_{\varepsilon}\left(  x\right)  $ of the point $x$.
\end{proposition}

\noindent\textbf{Proof} An implication is obvious as $U_{\varepsilon}\left(
x\right)  $ is a $\rho$-bounded set. Suppose now that the limit (\ref{eq:dupo}%
) holds for each sequence in some neighborhood $U_{\varepsilon}\left(
x\right)  $. Fix a sequence $t_{n}\downarrow0$ and any $\rho$-bounded sequence
$\left\{  y_{n}\right\}  \subseteq C$. From the relation
\[
\rho\left(  x,\left(  1-\alpha\right)  x+\alpha y_{n}\right)  \leq\alpha
\rho\left(  x,y_{n}\right)
\]
it follows that, for $\alpha>0$ sufficiently small, the sequence
$z_{n}=\left(  1-\alpha\right)  x+\alpha y_{n}$ belongs to $U_{\varepsilon
}\left(  x\right)  $ for all $n$. By setting $\tau_{n}=t_{n}/\alpha$, we can
write
\[
\lim_{n\rightarrow\infty}\left[  \frac{f\left(  x+\tau_{n}\left(
z_{n}-x\right)  \right)  -f\left(  x\right)  }{\tau_{n}}-A\left(
z_{n}\right)  \right]  =0
\]
Coming back to the old variables, we get the limit (\ref{eq:dupo}). Observe
that $A\left(  z_{n}\right)  =A\left(  \left(  1-\alpha\right)  x+\alpha
y_{n}\right)  =\alpha A\left(  y_{n}\right)  $.\hfill$\blacksquare$

\bigskip

Let us now introduce Frechet wa-differentiability. The convexity of the metric
$\rho$ is a convenient assumption as it will be readily seen (see Proposition
\ref{prop:fafa} below).

\begin{definition}
A function $f:C\rightarrow\mathbb{R}$ is \emph{Frechet wa-differentiable} at
$x\in C$ if there exists a $\rho$-continuous and affine functional
$A:C\rightarrow\mathbb{R}$, with $A\left(  x\right)  =0$, such that%
\begin{equation}
\lim_{n\rightarrow\infty}\frac{f\left(  x_{n}\right)  -f\left(  x\right)
-A\left(  x_{n}\right)  }{\rho\left(  x_{n},x\right)  }=0 \label{eq:hap}%
\end{equation}
for every sequence $\left\{  x_{n}\right\}  $ in $C$ with $\rho\left(
x_{n},x\right)  \rightarrow0$.
\end{definition}

We can also formulate a \textquotedblleft strict\textquotedblright\ version in
which%
\begin{equation}
\lim_{n\rightarrow\infty}\frac{f\left(  y_{n}\right)  -f\left(  x_{n}\right)
-A\left(  y_{n}\right)  }{\rho\left(  x_{n},y_{n}\right)  }=0 \label{eq:strfr}%
\end{equation}
holds for all $\rho\left(  x_{n},x\right)  \rightarrow0$ and $\rho\left(
y_{n},x\right)  \rightarrow0$.

\begin{proposition}
\label{prop:fafa}Let $\rho$ be convex. If $f:C\rightarrow\mathbb{R}$ is
Frechet wa-differentiable at $x\in C$, then it is boundedly wa-differentiable
at $x$.
\end{proposition}

\noindent\textbf{Proof} Assume (\ref{eq:hap}). Given a sequence $t_{n}%
\downarrow0$ and any bounded sequence $\left\{  x_{n}\right\}  $, set
$z_{n}=\left(  1-t_{n}\right)  x+t_{n}x_{n}$. Thanks to the convexity of
$\rho$, we have $\rho\left(  x,z_{n}\right)  \leq t_{n}\rho\left(
x,x_{n}\right)  $. Therefore, $\rho\left(  x,z_{n}\right)  \rightarrow0$. In
light of (\ref{eq:hap}),%
\[
f\left(  z_{n}\right)  -f\left(  x\right)  -A\left(  z_{n}\right)
=\rho\left(  z_{n},x\right)  o\left(  1\right)
\]
and
\[
f\left(  \left(  1-t_{n}\right)  x+t_{n}x_{n}\right)  -f\left(  x\right)
-t_{n}A\left(  x_{n}\right)  =t_{n}\left[  \frac{\rho\left(  z_{n},x\right)
}{t_{n}}\right]  o\left(  1\right)
\]
But $t_{n}^{-1}\rho\left(  z_{n},x\right)  \leq\rho\left(  x,x_{n}\right)  $,
which is bounded. Therefore,%
\[
f\left(  \left(  1-t_{n}\right)  x+t_{n}x_{n}\right)  -f\left(  x\right)
-t_{n}A\left(  x_{n}\right)  =o\left(  t_{n}\right)
\]
Hence, $f$ is boundedly wa-differentiable at $x$.\hfill$\blacksquare$

\bigskip

A consequence of this proposition is that the affine function $A$ in
(\ref{eq:hap}) is unique (if exists), with $D_{F}f\left(  x;\cdot\right)
=A\left(  \cdot\right)  $. Moreover, a Frechet wa-differentiable function
$f:C\rightarrow\mathbb{R}$ is Hadamard wa-differentiable and $D_{F}f\left(
x;\cdot\right)  =D_{H}f\left(  x;\cdot\right)  $.

Condition (\ref{eq:hap}) is like the classical Frechet condition for functions
on normed spaces, in which a distance replaces the norm. In fact,
(\ref{eq:hap}) can be written
\[
f\left(  y\right)  =f\left(  x\right)  +D_{F}f\left(  x;y\right)  +o\left(
\rho\left(  x,y\right)  \right)  \text{ \ as }\rho\left(  x,y\right)
\rightarrow0.
\]
When $D_{F}f\left(  x;\cdot\right)  $ is extendable, it becomes
\[
f\left(  y\right)  =f\left(  x\right)  +\left\langle \nabla_{F}f\left(
x\right)  ,y-x\right\rangle +o\left(  \rho\left(  y,x\right)  \right)
\]

Notice that there is no converse of Proposition \ref{prop:fafa} as it will be
momentarily seen.

\begin{example}
\emph{The quadratic functional }$Q:\Delta\left(  Y\right)  \rightarrow
\mathbb{R}$ \emph{given by}%
\begin{equation}
Q\left(  \mu\right)  =\int\psi\left(  x,y\right)  \mathrm{d}\mu\left(
x\right)  \otimes\mathrm{d}\mu\left(  y\right)  \label{eq:HHH}%
\end{equation}
\emph{is boundedly a-differentiable for any metric. Indeed this result is more
general. Consider a quadratic functional }$Q\left(  x\right)  =B\left(
x,x\right)  $ \emph{where }$B$\emph{ is a bilinear and symmetric form. For}
\emph{any sequence} $\left\{  y_{n}\right\}  $ \emph{and} $t_{n}\downarrow0$
\emph{we have}
\[
\frac{Q\left(  \left(  1-t_{n}\right)  x+t_{n}y_{n}\right)  -Q\left(
x\right)  }{t_{n}}-DQ\left(  x;y_{n}\right)  =t_{n}Q\left(  x\right)
+t_{n}Q\left(  y_{n}\right)  -2t_{n}B\left(  x,y_{n}\right)
\]
\emph{This quantity vanishes as }$n\rightarrow\infty$\emph{ provided
}$Q\left(  y_{n}\right)  $ \emph{and }$B\left(  x,y_{n}\right)  $ \emph{remain
bounded along any bounded sequence }$\left\{  y_{n}\right\}  $, \emph{as in
our case (\ref{eq:HHH})}. \emph{This fact does not imply that (\ref{eq:HHH})}
\emph{is always Frechet wa-differentiable. }

\emph{Let us show that the functional (\ref{eq:HHH})} \emph{is Frechet
a-differentiable for the variation metric. Indeed,}%
\[
Q\left(  \lambda_{n}\right)  -Q\left(  \mu\right)  -DQ\left(  \mu;\lambda
_{n}\right)  =Q\left(  \lambda_{n}\right)  +Q\left(  \mu\right)  -2B\left(
\mu,\lambda_{n}\right)
\]
\emph{Moreover, as }$Q$\emph{ can be extended to the whole space of signed
measures, we have further}
\[
Q\left(  \lambda_{n}\right)  -Q\left(  \mu\right)  -DQ\left(  \mu;\lambda
_{n}\right)  =Q\left(  \mu-\lambda_{n}\right)
\]
\emph{and so}
\begin{align*}
\left\vert Q\left(  \lambda_{n}-\mu\right)  \right\vert  &  =\left\vert
\int\psi\left(  x,y\right)  \mathrm{d}\left(  \lambda_{n}-\mu\right)
\otimes\mathrm{d}\left(  \lambda_{n}-\mu\right)  \right\vert =\left\vert
\int\mathrm{d}\left(  \lambda_{n}-\mu\right)  (y)\int\psi\left(  x,y\right)
\mathrm{d}\left(  \lambda_{n}-\mu\right)  \left(  x\right)  \right\vert \\
&  \leq\left\Vert \psi\right\Vert _{\infty}\left\vert \int\mathrm{d}\left(
\lambda_{n}-\mu\right)  (y)\int\mathrm{d}\left(  \lambda_{n}-\mu\right)
\left(  x\right)  \right\vert \leq\left\Vert \psi\right\Vert _{\infty
}\left\Vert \lambda_{n}-\mu\right\Vert _{V}^{2}%
\end{align*}
\emph{This implies}
\[
\frac{Q\left(  \lambda_{n}\right)  -Q\left(  \mu\right)  -DQ\left(
\mu;\lambda_{n}\right)  }{\rho_{V}\left(  \lambda_{n},\mu\right)  }%
\rightarrow0
\]
\emph{as} $\rho_{V}\left(  \lambda_{n},\mu\right)  \rightarrow0.$ \emph{Hence
(\ref{eq:hap}) holds.}\hfill

\emph{The functional (\ref{eq:HHH}) may fail to be Frechet wa-differentiable.
Take }$Y=\mathbb{R}$ \emph{and} \emph{ }$\psi\left(  x,y\right)  =f\left(
x\right)  \vee f\left(  y\right)  $\emph{,} \emph{where} $f:\mathbb{R}%
\rightarrow\mathbb{R}$ \emph{is non-constant and differentiable. Consider the
Prokhorov metric }$\rho_{P}$\emph{.} \emph{Pick a point} $x\in\mathbb{R}$
\emph{for which} $f^{\prime}\left(  x\right)  \neq0$\emph{, say }$f^{\prime
}\left(  x\right)  >0$ \emph{(the case }$f^{\prime}\left(  x\right)  <0$
\emph{is similar). Let }$y>x$\emph{ be sufficiently close to }$x$\emph{.
Then,}%
\[
Q\left(  \delta_{y}\right)  -Q\left(  \delta_{x}\right)  -DQ\left(  \delta
_{x};\delta_{y}\right)  =Q\left(  \delta_{y}-\delta_{x}\right)  =f\left(
x\right)  +f\left(  y\right)  -2f\left(  y\right)  =f\left(  x\right)
-f\left(  y\right)
\]
\emph{Since }$\rho_{P}\left(  \delta_{y},\delta_{x}\right)  =\min\left\{
\left\vert y-x\right\vert ;1\right\}  $\emph{, it follows that }%
\[
\lim_{x\downarrow x}\frac{Q\left(  \delta_{y}-\delta_{x}\right)  }{\rho
_{P}\left(  \delta_{y},\delta_{x}\right)  }=\lim_{x\downarrow x}\frac{f\left(
x\right)  -f\left(  y\right)  }{y-x}=-f^{\prime}\left(  x\right)  \neq0
\]
\emph{Hence, }$Q$\emph{ is not Frechet wa-differentiable at }$\delta_{x}%
$\emph{.}\hfill$\blacktriangle$
\end{example}

\begin{example}
\emph{Define} \emph{the convex and quadratic functional} $T:\mathcal{D}%
\rightarrow\mathbb{R}$ \emph{by}%
\[
T\left(  F\right)  =\int\left[  F\left(  x\right)  -F_{0}\left(  x\right)
\right]  ^{2}\mathrm{d}F_{0}\left(  x\right)
\]
\emph{where} $\mathcal{D}$ \emph{is the collection of all} \emph{probability
distributions on} $\mathbb{R}$. \emph{Given an empirical distributions }%
$F_{n}$\emph{, }$T\left(  F_{n}\right)  $\emph{ is the }Cramer-von Mises test
statistic\emph{ for the test problem }$H_{0}$\emph{: }$F=F_{0}$\emph{ versus
}$H_{1}$\emph{: }$F\neq F_{0}$\emph{.}

\emph{The wa-differential of }$T$ \emph{is}%
\[
DT\left(  F;G\right)  =2\int\left(  F-F_{0}\right)  \left(  G-F\right)
\mathrm{d}F_{0}\qquad\forall G\in\mathcal{D}%
\]
\emph{Notice that }$DT\left(  F_{0};\cdot\right)  =0$\emph{. Indeed, }$F_{0}$
\emph{minimizes }$T$\emph{. Let us show that, under the Kolmogorov metric
}$\rho_{\infty},$\emph{ }$T$\emph{ is strictly Frechet wa-differentiable at
every point }$F$\emph{ (see also \cite{shao}). Let }$\rho_{\infty}\left(
F_{n},F\right)  \rightarrow0$\emph{ and }$\rho_{\infty}\left(  G_{n},F\right)
\rightarrow0$\emph{ be two sequences. Some tedious algebra leads to}%
\[
R=T\left(  G_{n}\right)  -T\left(  F_{n}\right)  -DT\left(  F_{n}%
;G_{n}\right)  =\int\left(  G_{n}-F_{n}\right)  ^{2}\mathrm{d}F_{0}%
\]
\emph{Hence,}
\[
\left\vert R\right\vert \leq\left[  \rho_{\infty}\left(  G_{n},F_{n}\right)
\right]  ^{2}\leq\rho_{\infty}\left(  G_{n},F_{n}\right)  \left[  \rho
_{\infty}\left(  G_{n},F\right)  +\rho_{\infty}\left(  F,F_{n}\right)
\right]  =o\left(  \rho_{\infty}\left(  G_{n},F_{n}\right)  \right)
\]
\emph{which is (\ref{eq:strfr}).}\hfill$\blacktriangle$
\end{example}

Interpret the functional $\theta=T\left(  F\right)  $ as the parameter of an
unknown population, more generally $\theta=T\left(  \mu\right)  ,$ with
$\mu\in\Delta\left(  Y\right)  $. Inferences about $\theta$ are usually based
on the statistic $\hat{\theta}=T\left(  F_{n}\right)  $, where $F_{n}$ is the
empirical distribution%
\[
F_{n}=\frac{1}{n}\sum_{i=1}^{n}G_{x_{i}}%
\]
If $T$ is a-differentiable, we have (see (\ref{eq:FFT}))
\[
T\left(  F_{n}\right)  =T\left(  F\right)  +\left\langle \nabla_{a}T\left(
F\right)  ,F_{n}-F\right\rangle +R\left(  F;F_{n}\right)
\]
where $R\left(  F;F_{n}\right)  $ is the remainder. By setting $\nabla
_{a}T\left(  F\right)  =u_{F}\left(  x\right)  $ (the influence function,
i.e., the normalized gradient), we obtain
\[
T\left(  F_{n}\right)  =T\left(  F\right)  +\int u_{F}\mathrm{d}F_{n}+R\left(
F;F_{n}\right)  =T\left(  F\right)  +\frac{1}{n}\sum_{i=1}^{n}u_{F}\left(
x_{i}\right)  +R\left(  F;F_{n}\right)
\]
Suppose that an appropriately defined notion of differentiability guarantees
that%
\[
R\left(  F;F_{n}\right)  =o_{p}\left(  \frac{1}{\sqrt{n}}\right)
\]
The normalization ensures that the mean value\emph{ }$\int u_{F}\left(
x\right)  \mathrm{d}F\left(  x\right)  $ is zero$.$ By assuming that
$\sigma^{2}=\int u_{F}^{2}\left(  x\right)  \mathrm{d}F\left(  x\right)  $ is
finite, then%
\[
\sqrt{n}\left[  T\left(  F_{n}\right)  -T\left(  F\right)  \right]  =\frac
{1}{\sqrt{n}}\sum_{i=1}^{n}u_{F}\left(  x_{i}\right)  +o_{p}\left(  1\right)
\]
Slutsky's Lemma and the central limit theorem imply the asymptotic normality,
i.e.,
\[
\sqrt{n}\left[  T\left(  F_{n}\right)  -T\left(  F\right)  \right]
\rightarrow_{d}\mathcal{N}\left(  0,\sigma^{2}\right)
\]

Next we consider a simple case (see, e.g., \cite{fern}).

\begin{proposition}
Assume that for some metric $\rho$ we have
\[
\rho\left(  F_{n},F\right)  =O_{p}\left(  \frac{1}{\sqrt{n}}\right)
\]
If $T$ is Frechet differentiable at $F$, with normalized gradient $\nabla
T\left(  F\right)  =u_{F}$, then%
\[
\sqrt{n}\left[  T\left(  F_{n}\right)  -T\left(  F\right)  \right]
\rightarrow_{d}\mathcal{N}\left(  0,\sigma^{2}\right)
\]

\end{proposition}

\noindent\textbf{Proof} By (\ref{eq:hap}),%
\begin{align*}
T\left(  F_{n}\right)  -T\left(  F\right)   &  =\int u_{F}dF_{n}+\rho\left(
F_{n},F\right)  o\left(  1\right) \\
\sqrt{n}\left[  T\left(  F_{n}\right)  -T\left(  F\right)  \right]   &
=\frac{1}{\sqrt{n}}\sum_{i=1}^{n}u_{F}\left(  x_{i}\right)  +\frac{\sqrt
{n}O_{p}(1)o\left(  1\right)  }{\sqrt{n}}=\frac{1}{\sqrt{n}}\sum_{i=1}%
^{n}u_{F}\left(  x_{i}\right)  +o\left(  1\right)
\end{align*}
as desired.\hfill$\blacksquare$

\end{document}